\documentclass[12pt]{article}
\usepackage{xspace}
\usepackage[french]{babel}
\usepackage[latin1]{inputenc}
\usepackage[T1]{fontenc}
\usepackage{pslatex}
\usepackage{amsmath,amssymb,amsfonts,amsthm}
\usepackage[fixlanguage]{babelbib}
\selectbiblanguage{french}
\usepackage{comment}
\usepackage{imakeidx}
\usepackage{hyperref}%
\usepackage{fullpage}
\DeclareMathOperator{\Gl}{GL}

\def\ZZ{\ensuremath{\mathbb Z}}

\def\QQ{\ensuremath{\mathbb Q}}
\def\CC{\ensuremath{\mathbb C}}

\def\Qp{{\ensuremath{\QQ_p}}}
\def\Cp{\ensuremath{\CC_p}}
\def\Zp{\ensuremath{\ZZ_p}}

\def\gL{\ensuremath{\mathfrak L}}
\def\gX{\ensuremath{\mathfrak X}}
\def\bL{\ensuremath{\mathbf L}}

\def\cB{\ensuremath{\mathcal B}}
\def\cD{\ensuremath{\mathcal D}}
\def\cF{\ensuremath{\mathcal F}}
\def\cH{\ensuremath{\mathcal H}}

\def\limproj#1{\displaystyle{\lim_{\underset{n}{
\leftarrow}}}\
}

\DeclareMathOperator{\ord}{ord}
\DeclareMathOperator{\Gal}{Gal}

\DeclareMathOperator{\Fil}{Fil}
\theoremstyle{plain}
\newtheorem{thm}{Théorème}[section]
\newtheorem{thmb}{Théorème}
\newtheorem{lem}[thm]{Lemme}
\newtheorem{prop}[thm]{Proposition}
\newtheorem{cor}[thm]{Corollaire}
\theoremstyle{definition}
\newtheorem{defn}[thm]{Définition}
\newtheorem{rem}[thm]{Remarque}
\newtheorem*{reme}{Remarque}

\def\smallmat#1{\left(\begin{smallmatrix}#1\end{smallmatrix}\right)}
\def\id{\mathrm{Id}}

\def\lgg#1{\##1}
\def\inclus#1#2{{#1 \subset #2}}
\def\bomega#1#2{\Omega_{#1}^{#2}}
\def\txin#1#2{\Xi_{#1}^{#2}}
\def\mxiinf#1#2{\gX_{#1}^{#2}}
\def\mlog#1#2{\ell_{#2}^{#1}}
\def\mloginf#1#2{\textrm{Log}_{#2}^{#1}}
\def\JJ{J}
\def\Jp{J'}
\def\newtH{t'_{\HT,d}}
\def\newtHd{t'_{\HT,2}}
\def\II#1#2{]#1,#2]}
\def\III#1{]#1,t_{\HT,d}]}
\def\IIIprim#1{]#1,\newtH]}
\def\IIII#1{]#1,t_{\HT,3}]}
\def\ZZZ#1#2#3#4{P^{#1}_{#2,#4,#3}}
\def\Zinftyu#1#2#3{\textrm{Per}^{#1}_{#2,#3}}
\def\Module#1#2#3#4{\gL^{#1}_{#2,#3}}
\def\LLL#1#2#3{\bL_{#2}(#3)}
\def\TH#1#2#3{T_{#1,#2}^{#3}}
\def\Zi{H^1_{Iw}}
\def\DD{{\mathcal{D}}}

\def\np{n+1}
\def\nm{-(n+1)}
\def\phiun{(1\otimes\varphi)}

\def\HT{\textrm{HT}}
\def\Ne{\textrm{N}}

\def\norm#1#2{|| #1 ||_{#2}}
\def\abs#1#2{| #1 |_{#2}}
\def\res#1{M_{#1}}
\def\UU#1#2{{\varphi_{#1}^{#2}}}
\def\VV#1{\varphi_{#1}}
\def\Er#1#2{\mathcal{E}_{#2}(#1)}
\def\tborne#1{\tilde{\omicron}_p(#1)}
\def\borne#1{\omicron_p(#1)}
\def\coeffZ#1#2{Z_{#2,#1}}
\def\sm#1{t_{sm,#1}}
\def\slope#1{sl(#1)}
\def\EEE{\Qp}
\def\ot{_{\Qp}}
\def\oEE{{_{\EEE}}}
\def\oEE{{}}
\def\oZE{{{\Zp}}}

\def\Gg{g}
\def\Reg#1{\textrm{Reg}_{#1}}
\def\rabiot{\epsilon}
\def\U{\iota}
\usepackage[noblocks]{authblk}
\title{Note sur les périodes d'Iwasawa associées à un $\varphi$-module filtré}
\author{Bernadette Perrin-Riou}
\affil{Université Paris-Saclay, CNRS, Laboratoire de mathématiques d'Orsay, 91405, Orsay, France.}
\begin{document}
\maketitle
Soit $p$ un nombre premier impair. Soit $K_n=\Qp(\mu_{p^n})$ le corps de définition
des racines $p^n$-ièmes de l'unité et $K_\infty$ la réunion des $K_n$.
Soit $\DD$ un $\varphi$-module filtré sur $\Qp$ c'est-à-dire
un $\Qp$-espace vectoriel de dimension finie $d$ muni d'un endomorphisme $\varphi$ bijectif
et d'une filtration décroissante exhaustive et séparée $\Fil^\bullet \DD$.
Soit $\cH$ l'anneau des séries convergentes sur toute boule
$B(\rho) = \{x \in \Cp \;\textrm{ tel que } \;\abs{ x}{p} \leq \rho\}$
pour $\rho < 1$, $\Lambda=\Zp[[x]]$ et $\Lambda\ot =\Qp \otimes \Lambda$. On note pour $\rho < 1$
$$\norm{f}{\rho}= \sup_{\abs{x}{p}=\rho} \abs{f(x)}{p}=\sup_{\abs{x}{p}\leq \rho} \abs{f(x)}{p}$$

L'étude des représentations $p$-adiques cristallines sur $\Qp$ en théorie d'Iwasawa
conduit à définir, \emph{uniquement par des conditions analytiques},
des sous-$\cH$-modules naturels $\LLL{}{N}{\DD}$ de $\cH \otimes \DD$
où $\DD$ est le $\varphi$-module filtré sur $\Qp$ associé à la représentation $p$-adique.
\cite{jams}
Dans cet article, nous les étudions sans référence aux représentations $p$-adiques
et en construisons une base "explicite". Il s'agit d'une étude purement analytique.

Fixons un générateur $u$ de $1+p\Zp$.
Soit $t_{H,d}$ l'entier tel que $\Fil^{t_{H,d}}\DD\neq 0$, $\Fil^{t_{H,d}+1}\DD=0$.
Si $N$ est un entier $\geq 0$,
définissons $\LLL{}{N}{\DD}$ comme le sous-$\cH$-module de
$\cH \otimes \DD$ formé des
$g$ tels que pour tout entier $j \leq t_{H,d}$ et pour $n \geq N$
\begin{equation}
g(u^{j}\zeta-1) \in (1\otimes\varphi)^{\np}K_n \otimes \Fil^{j} \DD
\end{equation}
pour toute racine de l'unité $\zeta$ d'ordre $p^{n}$.
Comme $\Fil^{j} \DD=\DD$ pour $j \ll 0$, il n'y a qu'un nombre fini de conditions.
Une notion importante est la notion de \emph{pente} : si $f$ est un élément de
$\cH \otimes \DD$, on dit que $f$ est de pente $\leq s$ si
la suite des $\norm{p^{n(s+t_{H,d})} (1\otimes \varphi)^{-n} f }{\rho^{1/p^n}}$
est bornée pour un $\rho <1$. Nous nous intéressons au sous-$\Lambda_{\Qp}$-module
engendré par les éléments de $\LLL{}{N}{\DD}$ qui sont de pente $0$.

Si  $\JJ$ est un intervalle fini de $\ZZ$, on pose
 $$\mloginf{\JJ}{0}(x)= \prod_{j\in \JJ} \log(u^{-j} (1+x))
  =\prod_{j\in \JJ} \prod_{n=0}^\infty \frac{\xi_n(u^{-j}(1+x)-1)}{p}$$
  avec $\xi_n=\frac{(1+x)^{p^n}-1}{(1+x)^{p^{n-1}}-1}$ pour $n\geq 1$ et
  $\xi_0=px$
 et plus généralement
 $$\mloginf{\JJ}{N}(x)\index{$\mloginf{\JJ}{N}$}=\prod_{j\in \JJ}\prod_{n=N}^ \infty \frac{\xi_n(u^{-j}(1+x)-1)}{p}.$$

Nous aurons besoin d'introduire un entier $\delta\geq 0$ dépendant de la longueur
de la filtration de $\DD$
et nulle dès que cette longueur est strictement inférieure à $p$.
Rappelons que le $\varphi$-module filtré $\DD$ est dit \textsl{faiblement admissible}
s'il admet un réseau fortement divisible
(voir \ref{def:faibleadm} pour le rappel de la définition initiale).
\begin{thmb}
Soit $\DD$ un $\varphi$-module filtré de dimension $d$.
\begin{enumerate}
\item Le $\cH$-module $\LLL{}{N}{\DD}$ est de rang $d$.
\item On associe de manière explicite à toute base $\cB$ de $\DD$ adaptée
à la filtration $\Fil ^\bullet \DD$
une base $\left(\Zinftyu{}{\cB}{N}(v)\right)_{v\in \cB}$ de $\LLL{}{N}{\DD}$
formée d'éléments de pente finie.
Lorsque le $\varphi$-module filtré $\DD$ est faiblement admissible et que
$\cB$ est une base adaptée au $\varphi$-module filtré $\DD$ engendrant un réseau
fortement divisible de $\DD$, la base construite est formée d'éléments de pente positive et
inférieure ou égale à $\delta$.
\item Supposons $\delta=0$.
Le sous-$\Lambda\ot$-module de $\LLL{}{N}{\DD}$ engendré par les éléments de pente nulle
est engendré par les $\left(\Zinftyu{}{\cB}{N}(v)\right)_{v\in \cB}$.
Si $t_{H,1} \leq \cdots \leq t_{H,d}$ sont les nombres de Hodge de $\DD$ avec multiplicités,
la suite des diviseurs élémentaires de $\LLL{}{N}{\DD}$ dans
$\cH \otimes \DD$ est
$$[\mloginf{\III{t_{\HT,1}}}{N}; \cdots; \mloginf{\III{t_{\HT,j}}}{N}; \cdots ;
 1].$$
\item Si de plus $\cB$ est adaptée à un raffinement de $\cD$,
la base trouvée est une base adaptée à l'inclusion
du $\cH$-module $\LLL{}{N}{\DD}$ dans $\cH \otimes \DD$.
\end{enumerate}
\end{thmb}

Ces résultats peuvent être utilisés dans l'étude des régulateurs $p$-adiques d'une représentation
$p$-adique cristalline associée à un motif. Nous donnons dans cette introduction l'idée
générale, sans rentrer dans les détails.

Notons $G_\infty=\Gal(K_\infty/\Qp)=\Delta\times \Gamma$
avec $\Delta=\Gal(K_1/\Qp)$ et $\Gamma=\Gal(K_\infty/K_1)$
sa décomposition canonique,
$\chi$ le caractère cyclotomique de $G_\infty$ à valeurs dans $\Zp^\times$.
Soit $\gamma$ un générateur de $\Gamma$.

Soit $V$ une représentation $p$-adique cristalline sur $\Qp$ de dimension $d$
et $T$ un réseau de $V$ stable par $G_{\Qp}$.
Soit $\DD$ le $\varphi$-module filtré sur $\Qp$ faiblement admissible associé.
La théorie d'Iwasawa associée à $V$ et à l'extension cyclotomique $K_\infty$ de $\Qp$ mesure
le module d'Iwasawa
$$\Zi(\Qp, V)=\Qp\otimes \limproj{n}H^1(K_n,T)$$
à l'aide d'une application régulateur \footnote{Une longue histoire depuis
\cite{bpr93}, \cite{bpr-debut}, \cite{jams}, \cite{bpr-colmez}, \cite{semi-stable},
\cite{divelem}, nous ne parlerons pas ici de sa réinterprétation en termes de
$(\varphi,\Gamma)$-modules.
}
$$\Reg{V}: \Zi(\Qp, V) \to \cH(G_\infty) \otimes \DD$$
où
$\cH(G_\infty)=\Zp[\Delta]\otimes \cH(\Gamma)$
et $\cH(\Gamma)$ l'image de $\cH$ par l'application induite
par $x\to \gamma-1$.
On note $Tw$ l'opérateur de twist
sur $\cH(G_\infty)$ induit par $g \to \chi(g) g$ pour $g \in G_\infty$.
\footnote{Sans redéfinir l'opérateur $\psi$ ici, rappelons que l'on trouve
aussi une formulation avec $\cH^{\psi=0}$. On passe facilement de l'une
à l'autre par l'isomorphisme
$\cH(G_\infty) \to \cH^{\psi=0}$
induit par $g\to g\cdot (1+X)=(1+X)^{\chi(g)}$.
}
Si $\LLL{}{G_\infty}{\DD}$\index{$\LLL{}{G_\infty}{\DD}$} est
le sous-module de $\cH(G_\infty) \otimes \DD$
formé des éléments $f$ tels que
pour $j \leq t_{\HT,d}$
\begin{equation}
\eta(Tw^{j}(f))\in \varphi^{f_\eta} ( K_n \otimes \Fil^j\DD)
\end{equation}
pour tout caractère $\eta$ non trivial de conducteur $p^{f_\eta}$ avec $f_\eta \geq 1$
et une définition adaptée pour $f_\eta=0$,
l'image de $\Zi(\Qp, V)$ dans $\cH(G_\infty) \otimes \DD$ est le sous-$\Lambda\ot $-module de
$\LLL{}{G_\infty}{\DD}$ de pente 0.
Par l'isomorphisme de Mellin-Amice $\cH(G_\infty)\cong \Zp[\Delta]\otimes\cH$,
les composantes non triviales relatives à $\Delta$
du $\cH(G_\infty)$-module $\LLL{}{G_\infty}{\DD}$
sont égales aux composantes non triviales
du sous-module de $\Qp[\Delta]\otimes\bL(\DD)$
engendré sur
$\Qp[\Delta]\otimes \Zp[\Gamma]$
par l'image des
$\left(\Zinftyu{}{\cB}{0}(v)\right)_{v\in \cB}$ avec une adaptation dans le cas
de la composante triviale dûe au "facteur
d'Euler".

Lorsque $V$ provient d'une représentation $p$-adique du groupe de Galois
de $\Gal(\overline{\QQ}/\QQ)$ (ou d'un motif),
l'étude du régulateur d'Iwasawa de certains éléments spéciaux
globaux et leur comparaison avec les "périodes d'Iwasawa" introduites ici
permettent de construire des éléments de $\Qp\otimes \Zp[[G_\infty]]$
reliés aux fonctions $L$ $p$-adiques. D'où l'intérêt de comprendre
l'image de $\Zi(\Qp, V)$ dans $\cH(G_\infty) \otimes \DD$, de
même que dans le cas complexe, on écrit les valeurs des fonctions $L$ dans
des périodes complexes. Cela est déjà connu pour les fonctions $L$ $p$-adiques
associées aux courbes elliptiques, aux formes modulaires et à certaines représentations
s'en déduisant (le premier cas connu est dû à Pollack \cite{pollack}).

Donnons le plan du texte.
Dans la première partie, nous construisons, en utilisant les méthodes d'Amice-Vélu, pour  $\JJ$ intervalle fini de $\ZZ$,
des fonctions analytiques $\mxiinf{N}{\JJ}$
par interpolation  en des valeurs du type $u^j\zeta -1$ pour $j\in \JJ$,
$u\in 1+p\Zp$ et $\zeta$ une racine $p^n$-ième de l'unité. Ces fonctions sont divisibles par $\mloginf{\JJ}{N}$,
 en diffèrent par une unité lorsque la longueur de  $\JJ$ est inférieure à $p$. Comme
 $\mloginf{\JJ}{N}$, elles sont définies
 comme des produits convergents mais sont de vitesse de convergence plus grande.

Dans la seconde partie, après avoir rappelé des définitions sur les
$\varphi$-modules filtrés, nous construisons un système libre de rang maximal de
$\LLL{}{N}{\DD}$, associé à une base adaptée à la filtration.
Nous en calculons la pente dans $\cH\otimes \DD$.
Deux sortes de bases particulières se trouvent dans la littérature :
bases engendrant un réseau de $\DD$ fortement divisible, bases associées à un raffinement.
Nous faisons le calcul dans les deux cas car selon
la base choisie, les éléments définis ont des formules plus ou moins agréables et des
propriétés différentes.
En particulier, le calcul dans une base associée à un raffinement permet de mettre en évidence
la dépendance par rapport à la filtration.

Nous traitons plus particulièrement le cas de dimension 2,
autrement dit les résultats de Pollack \cite{pollack} établis dans le cas d'une
représentation $p$-adique associée à une courbe elliptique ayant bonne réduction
supersingulière pour $p > 3$.
Des résultats de ce type peuvent se déduire d'articles utilisant la
notion de $(\varphi,\Gamma)$-modules comme les articles de Lei, Loeffler, Zerbes
et autres (\cite{LLZ10}, \cite{LLZ11}, \cite{LZ}, \cite{llz17}).

Il est possible sans difficulté particulière de travailler avec un $\varphi$-modules filtré
à coefficients dans une extension finie de $\Qp$. Nous ne le faisons pas explicitement.
Nous ne développons pas non plus ici les conséquences relatives aux conjectures principales,
à la croissance de l'ordre du groupe de Tate-Shafarevich le long de l'extension cyclotomique,
etc.

\section{Construction d'éléments de $\cH$}
Dans la suite de ce paragraphe, $r$ désigne un entier strictement positif.
\subsection{Interpolation de polynômes}

On pose $\omega_n=(1+x)^{p^n} -1$ pour $n \geq 0$,
$\xi_n = \omega_n/\omega_{n-1}$\index{$\xi_n$} pour $n \geq 1$.
Pour $n=0$, on prend $\xi_0=p x$.
On a alors $\log(1+x)=\lim_{n\to \infty} \frac{\omega_n}{p^n}
=\prod_{s=0}^\infty \frac{\xi_s}{p}$.

Soit une unité $p$-adique $u$ congrue à 1 modulo $p$
(nous ne la mettons pas dans la notation des définitions qui suivent,
sauf besoin).
Si $f \in \cH$, on note $f^{(j)}(x)=Tw_u^{-j}(f)(x) = f(u^{-j}(1+x)-1)$\index{$f^{(j)}$}:
on a $f^{(j)}(u^j-1)=f(0)$ et
$Tw_u^{j}(f)(0)= f(u^j-1)$.

Ainsi,
\begin{equation*}
\begin{split}
\omega_n^{(j)}&=
\omega_n(u^{-j}(1+x)-1)=u^{-jp^n}(1+x)^{p^n}-1\\
\xi_n^{(j)}&= \sum_{k=0}^{p-1}u^{-jkp^{n-1}} (1+x)^{kp^{n-1}} \text{pour $n>0$}.
\end{split}
\end{equation*}

Rappelons la proposition suivante reprise des travaux d'Amice et Vélu (\cite{av}).
\begin{prop}
\label{av} Soit $n$ un entier $\geq 0$.
Soient $r$ polynômes $Q_0, \cdots, Q_{r-1}$ de degré strictement inférieur à $p^n$. Soit $P$ le
polynôme de degré $< rp^n$ tel que $$P \equiv Q_j\bmod \omega_n^{(j)}$$
pour $0\leq j<r$
et soit $$\delta_i = \sum_{j=0}^i (-1)^{i-j}\binom{i}{j}Q_j$$
pour $0\leq i<r$.
Soit $\rho$ un réel positif (strictement) inférieur à $\rho_0=p^{-\frac{1}{p-1}}$.
Alors,
\begin{equation}\label{const:maj}
  \norm{P}{1} \leq
  p^{\frac{r-1}{p-1}} \max_{0\leq i<r}
\left(\frac{p^{-\frac{i}{p-1}}}{\abs{i!}{p}}
\frac{\norm{\delta_i}{1}}{\abs{u^{p^n}-1}{p}^i}\right)
\leq
p^{\frac{r-1}{p-1}} \max_{0\leq i<r}
\left(\frac{\norm{\delta_i}{1}}{\abs{u^{p^n}-1}{p}^{i}}\right).
\end{equation}
\end{prop}
\begin{rem}
Si $Q_j$ est de la forme $Q^{(j)}$ pour un polynôme $Q$,
on a
$$\norm{\delta_i}{1}\leq \norm{Q}{1} \abs{u^{p^n}-1}{p}^i.$$
\end{rem}
\begin{proof}
Soit $S$ l'unique polynôme en $x$ et $y$ de degré $<r$ en $y$ et de degré $< p^n$
en $x$ tel que $S(x,u^{jp^n}-1) = Q_j(x)$ pour $0\leq j<r$.
Soit $Q(x,z)= \sum_{i=0}^{r-1}\delta_i(x)\binom{z}{i}$. On a
$Q(x,j)=Q_j(x)$ pour $0\leq j<r$ et
$$S(x,u^{jp^n}-1)=Q(x,j).$$
Comme $S(x,(1+x)^{p^n}-1)$ est de degré strictement inférieur à $rp^n$
et vérifie la condition demandée pour $P$, on a
$$P(x)=S(x,(1+x)^{p^n}-1).$$
Posons $S = \sum_{k=0}^{r-1} S_k(x) y^k$ vu comme polynôme en $y$.
On met sur l'espace vectoriel des polynômes en $y$ (ou en $z$)
à coefficients dans un espace muni d'une norme la norme obtenue en prenant le sup
de la norme des coefficients.
On a donc $\norm{S}{\rho}=\max_{k} \norm{S_k}{1} \rho^k$.
D'après \cite{av}, pour $\rho\leq \rho_0=p^{-1/(p-1)}$,
$z \mapsto v^z-1$ (avec $v={u^{p ^n}}$)
est un isomorphisme de la boule $\abs{z}{p} < \rho/\abs{v-1}{p}$
sur la boule $\abs{y}{p} < \rho$, l'isomorphisme réciproque étant
$ y \mapsto \frac{\log(1+y)}{\log v}$. Donc
$$\norm{S}{\rho} = \norm{ Q}{\rho/\abs{u^{p^n}-1}{p}}\ ,$$
d'où
$$\norm{S}{\rho}=\max_{0\leq k < r}(\norm{S_k}{1}\rho^k)
  = \max_{0\leq i <r}(\norm{\delta_i}{1}
\times \norm{\binom{\cdot}{i}}{\rho/\abs{u^{p^n}-1}{p}}).$$
Pour $\rho>\abs{u^{p^n}-1}{p}$, on a
$$\norm{\binom{\cdot}{i}}{\rho/\abs{u^{p^n}-1}{p}}=
\frac{(\rho/\abs{u^{p^n}-1}{p})^i}{\abs{i!}{p}}$$
et
$$\norm{S}{\rho}= \max_{0\leq i <r}\left(\norm{\delta_i}{1}
  \frac{(\rho/\abs{u^{p^n}-1}{p})^i}{\abs{i!}{p}}\right)
$$
D'où
\begin{equation*}
\norm{ P }{1} \leq \norm{ S }{1}\leq \rho^{-(r-1)}\max_{0\leq i<r}
\left(\frac{\norm{\delta_i }{1}}{\abs{u^{p^n}-1}{p}^i}
  \frac{\rho^i}{\abs{i!}{p}}\right)=
  \max_{0\leq i <r}\left(
\frac{\norm{\delta_i}{1}}{\abs{u^{p^n}-1}{p}^i}\frac{\rho^{i-r+1}}{\abs{i!}{p}}
\right)
\end{equation*}
En prenant $\rho=p^{-\frac{1}{p-1}}$, on en déduit la première inégalité de
\eqref{const:maj}. La deuxième inégalité se déduit de ce que
$$\ord_p(i!) = \sum_{s=1}^i\lfloor \frac{i}{p^s}\rfloor \leq
  \frac{i}{p} \frac{1-\frac{1}{p^i}}{1-\frac{1}{p}} \leq
\frac{i}{p-1}$$ pour $i\geq 0$
et donc que $\frac{p^{-\frac{i}{p-1}}}{\abs{i!}{p}}\leq 1$.
\end{proof}

\subsection{Construction de périodes logarithmiques dans $\cH$}
Nous introduisons dans ce paragraphe des éléments spéciaux de
$\cH$ que l'on peut voir comme une variante du produit de logarithmes.
Dans la suite,  $\JJ$ est un intervalle fini non vide de $\ZZ$.

On pose \index{$\bomega{n}{\JJ}$, $\mlog{\JJ}{n}$}
$\bomega{n}{\JJ}=\prod_{j\in J} \omega_n^{(j)}$\index{$\bomega{n}{\JJ}$, $\mlog{\JJ}{n}$},
$\mlog{\JJ}{n}=\prod_{j\in J} \frac{\xi_n^{(j)}}{p}$.
On a donc $\mloginf{\JJ}{0}=\prod_{n=0}^\infty \mlog{\JJ}{n}$.
Si $P$ et $Q$ sont deux polynômes premiers entre eux,
on désigne par $(P \bmod Q)^{-1}$ l'unique polynôme $R$ de degré strictement inférieur
au degré de $Q$ tel que $R P\equiv 1 \bmod Q$.
\begin{prop}\label{prop:existence}
Soit $\JJ$ un intervalle non vide de $\ZZ$ de cardinal $\lgg{\JJ}$.
Pour tout $n\geq 1$, soit
$\txin{n}{\JJ}$\index{$\txin{n}{\JJ}$} l'unique polynôme de $\Qp[x]$ de degré $< p^n\lgg{\JJ}$ tel que
$$\txin{n}{\JJ} \equiv \frac{\xi_n^{(j)}}{p} \bmod \omega_n^{(j)}$$
pour tout $j \in J$. Alors,
\begin{equation}\label{eq_xin}
\txin{n}{\JJ} =
(\mlog{\JJ}{n}\bmod \bomega{n-1}{\JJ})^{-1} \mlog{\JJ}{n}
\equiv \begin{cases}
0 & \bmod \bomega{n}{\JJ}/\bomega{n-1}{\JJ}\\
1 &\bmod \bomega{n-1}{\JJ}
\end{cases}
\end{equation}
et
$$p^{\lgg{\JJ}}\leq \norm{\txin{n}{\JJ}}{1} \leq p^{\lgg{\JJ}+ \borne{\lgg{\JJ}}}$$
avec $\borne{s} = \ord_p((s-1)!)$\index{$\borne{s}$}.
De plus, si $\borne{\lgg{\JJ}}=0$, $\txin{n}{\JJ}/\mlog{\JJ}{n}$ est une unité de $\Zp[[x]]$.
\end{prop}
\begin{proof}
L'existence vient du lemme chinois.
L'identité \eqref{eq_xin} se déduit de ce que les valeurs
des polynômes $\txin{n}{\JJ}$ et $ (\mlog{\JJ}{n} \bmod {\bomega{n-1}{\JJ}})^{-1} \mlog{\JJ}{n}$
 en $\lgg{\JJ}p^n$ points sont les mêmes
(leurs valeurs en $u^{j}\zeta-1$ pour $\zeta$ d'ordre $p^m$
sont 0 pour $m= n$ et $1$ pour $m<n$) et que leurs degrés
sont tous deux strictement inférieurs à $p^n\lgg{\JJ}$ (le degré du deuxième
est strictement inférieur à $p^{n-1}\lgg{\JJ} + (p-1)p^{n-1}\lgg{\JJ}=p^n\lgg{\JJ}$).
Pour calculer $\norm{\txin{n}{\JJ}}{1}$,
supposons que  $\JJ=[0,r[$ (ce qui est possible par
translation) et appliquons la proposition \ref{av} aux polynômes
 $Q_j = \xi_n^{(j)}/p$ pour $j = 0, \dots , r-1$.
Pour $0\leq i <r$, on a
\begin{equation*}
\begin{split}
\delta_i(x) &=p^{-1}\sum_{j=0}^{i}\sum_{s=0}^{p-1} (-1)^{i-j}\binom{i}{j}
u^{-sjp^{n-1}} (1+x)^{sp^{n-1}}
\\&=p^{-1}\sum_{s=0}^{p-1}\sum_{j=0}^i (-1)^{i-j}\binom{i}{j}
u^{-sjp^{n-1}} (1+x)^{sp^{n-1}}
\\&= p^{-1}\sum_{s=0}^{p-1}(u^{-sp^{n-1}} -1)^i (1+x)^{sp^{n-1}}
.
 \end{split}
 \end{equation*}
Donc
\begin{equation*}
\norm{\delta_i}{1} \leq
p\max_{1\leq s \leq p-1}\abs{u^{-sp^{n-1}}-1}{p}^i
=\abs{u^{p^n}-1}{p}^ip^{1+i}.
\end{equation*}
On en déduit que
$$
  \norm{\txin{n}{\JJ} }{1} \leq
  p^{\frac{r-1}{p-1}} \max_{0\leq i<r}
\left( \frac{p^{1+i-\frac{i}{p-1}}}{\abs{i!}{p}}\right)
=
  p^{\frac{r-1}{p-1}}
\frac{p^{r-\frac{r-1}{p-1}}}{\abs{(r-1)!}{p}}= \frac{p^r}{\abs{(r-1)!}{p}}.
$$
Donc, $$
\norm{\txin{n}{\JJ} }{1} \leq p^{r+\ord_p((r-1)!}=p^{\lgg{\JJ}+\borne{\lgg{\JJ}}}.$$
Finalement, il est clair que
$\txin{n}{\JJ}$ est divisible par $\mlog{\JJ}{n}$
et que $\norm{\mlog{\JJ}{n}}{1}=p^{\lgg{\JJ}}$.
Pour tout entier $l$,
$\frac{\xi_n(u^{l} -1)}{p}$ est une unité $p$-adique
: pour $l=0$, $\frac{\xi_n(u^{l} -1)}{p}=1$, pour $l\neq 0$,
$\frac{\xi_n(u^{l} -1)}{p}=\frac{u^{lp^{n}} -1}{p(u^{lp^{n-1}} -1)}$ est de valuation
$p$-adique $0$. On en déduit que
pour $j\in J$ et $n\geq 1$,
la valeur de $\txin{n}{\JJ}/\mlog{\JJ}{n}$ en $u^j-1$ est une unité
et donc que
$\norm{\txin{n}{\JJ}/\mlog{\JJ}{n}}{1} \geq 1$.
Lorsque $\borne{\lgg{\JJ}}=0$, $\txin{n}{\JJ}/\mlog{\JJ}{n}$ appartient à
$\Zp[[x]]$ et en est une unité car une de ses valeurs dans le disque unité ouvert
est une unité.
\end{proof}
La proposition suivante généralise la proposition \ref{prop:existence}.
\begin{prop}
\label{const:norme}
Soit  $\Jp$ un sous-intervalle de  $\JJ$.
Soit $n$ un entier supérieur ou égal à 1.
Soit $\txin{n}{\inclus{ \Jp}{\JJ}}$\index{$\txin{n}{\inclus{ \Jp}{\JJ}}$} l'unique polynôme de $\Qp[x]$ de degré
strictement inférieur à
$((p-1) \lgg{ \Jp} + \lgg{\JJ})p^{n-1}$ vérifiant
$$\txin{n}{{\inclus{ \Jp}{\JJ}}} \equiv
\begin{cases}
\frac{\xi_n^{(j)}}{p} \bmod \omega_n^{(j)} &\text{pour $j \in  \Jp$}\\
1 \bmod \omega_{n-1}^{(j)} &\text{pour $j \in \JJ -  \Jp$}.
\end{cases}
$$
Alors, $$
\txin{n}{\inclus{ \Jp}{\JJ}} = (\mlog{ \Jp}{n} \bmod \bomega{n-1}{\JJ})^{-1} \mlog{ \Jp}{n}
\equiv\begin{cases}
0 &\bmod \bomega{n}{ \Jp}/\bomega{n-1}{\Jp}\\
1 &\bmod \bomega{n-1}{\JJ}
\end{cases}
$$
et $$p^{\lgg{\Jp}}\leq \norm{\txin{n}{\inclus{\Jp}{\JJ}}}{1} \leq p^{\lgg{\Jp}+\lfloor\frac{\lgg{\JJ}-1}{p-1} + \max(0, \frac{1}{p-1}+\ord_p((\lgg{\JJ}-1)!)
-\ord_p(u^{p^{n-1}}-1))\rfloor}.$$
En particulier, pour $n > \ord_p((\lgg{\JJ}-1)!) - (\ord_p(u-1)-1)$,
$$p^{\lgg{\Jp}}\leq \norm{\txin{n}{\inclus{\Jp}{\JJ}}}{1} \leq p^{\lgg{\Jp}+\tborne{\lgg{\JJ}}}$$
avec $\tborne{s}=\lfloor\frac{s-1}{p-1}\rfloor$.\index{$\tborne{s}$}
Si de plus $\tborne{\lgg{\JJ}}=0$,
  $\txin{n}{\inclus{\Jp}{\JJ}}/\mlog{\Jp}{n}$ est une unité de $\Zp[[x]]$.
\end{prop}
\begin{rem}
\begin{enumerate}
\item Lorsque $\ord_p(u-1)=1$, la condition est que $n > \ord_p((\lgg{\JJ}-1)!)$.
\item
Si  $\Jp=\JJ$, $\txin{n}{\inclus{\JJ}{\JJ}}=\txin{n}{\JJ}$. Si  $\Jp=\emptyset$,
on a $\txin{n}{\inclus{\emptyset}{\JJ}}=1$.
\item Le degré de $\txin{n}{\inclus{\Jp}{\JJ}}$
est en fait inférieur ou égal à $((p-1)\lgg{\Jp} + \lgg{\JJ} -1)p^{n-1}$
pour $n\geq 1$.
\item On prend $\txin{0}{\inclus{\Jp}{\JJ}}=\mlog{\Jp}{0}$.
\end{enumerate}
\end{rem}
\begin{proof} Comme précédement, il est clair que
$\txin{n}{\inclus{\Jp}{\JJ}}$ est divisible par $\mlog{\Jp}{n}$
et que
$\norm{\txin{n}{\inclus{\Jp}{\JJ}}/\mlog{\Jp}{n}}{1}\geq p^{\lgg{\JJ}'}$.
Démontrons que pour $n\geq 1$,
\begin{equation}\label{majx}
\norm{\txin{n}{\inclus{\Jp}{\JJ}}/\mlog{\Jp}{n}}{1}\leq
p^{\frac{\lgg{\JJ}-1}{p-1}}\max(1,\frac{p^{\frac{1}{p-1}}\abs{u^{p^{n-1}}-1}{p}}{\abs{(\lgg{\JJ}-1)!}{p}}).
\end{equation}
Il suffit de regarder le cas où $n=1$ et $u$ quelconque congru à 1 modulo $p$. En effet,
$\txin{n}{\inclus{\Jp}{\JJ}}$ est obtenu en appliquant
$\varphi^{n-1} : x \mapsto (1+x)^{p^{n-1}} -1$
à $\txin{1}{\inclus{\Jp}{\JJ}}$ et en remplaçant $u$ par $u^{p^{n-1}}$:
$$\txin{n,u}{\inclus{\Jp}{\JJ}}(x)
=\txin{1,u^{p^{n-1}}}{\inclus{\Jp}{\JJ}}(\varphi^{n-1}(x)),\quad
\mlog{\Jp}{n,u}=\mlog{\Jp}{1,u^{p^{n-1}}}(\varphi^{n-1}x)$$
puisque $\varphi$ conserve la boule unité et on a donc
$$\norm{\txin{n,u}{\inclus{\Jp}{\JJ}}/\mlog{\Jp}{n,u}}{1}
=\norm{\txin{1,u^{p^{n-1}}}{\inclus{\Jp}{\JJ}}/\mlog{\Jp}{1,u^{p^{n-1}}}}{1}.$$
Posons $r=\lgg{\JJ}$. On peut supposer que  $\JJ=[0,\cdot\cdot,r[$
quitte à faire une translation.
Remarquons que $\mlog{\Jp}{1}$ est
inversible dans le disque $B(0, \rho)$ avec $\rho < \rho_0$
et que $\txin{1}{\inclus{\Jp}{\JJ}}/\mlog{\Jp}{1}$ est par définition le polynôme $V$ de
degré strictement inférieur à $\lgg{\JJ}$
vérifiant $V(u^j-1)= \mlog{\Jp}{1}(u^j-1)^{-1}$ pour $j\in J$.

La fonction $f=\frac{1}{\mlog{\Jp}{1}}=\sum_{m=0}^\infty a_m x^m$
est une fonction analytique sur $B(0,\rho_0)$.
Posons
$\alpha_i=u^j-1$ pour $j\in J$. Ce sont des éléments
de la boule $B(0,\abs{u-1}{p})$.
Le polynôme $V$ est le polynôme d'interpolation des $f(\alpha_i)$ en les $\alpha_i$.
Le rayon de convergence de $f$ étant $\rho_0$, on a
$$\abs{a_m}{p}\leq \rho_0^{-m} \norm{f}{\rho_0}.$$
Écrivons $f=Q+x^r g$ avec $Q$ polynôme de degré $\leq r-1$.
On a
$$
\abs{f(\alpha_i) - Q(\alpha_i)}{p}
\leq \max_{m\geq r} \abs{a_m}{p}\abs{u-1}{p}^m
\leq \left(\frac{\abs{u-1}{p}}{\rho_0}\right)^{r} \norm{f}{\rho_0}.
$$
Soit $P$ le polynôme interpolant les $f(\alpha_i) - Q(\alpha_i)$ en $\alpha_i$
pour $i=0,\cdots,r-1$. On a donc $V=P+Q$.
De la formule
$$
  P=\sum_{i=0}^{r-1} \left(f(\alpha_i) - Q(\alpha_i)\right)
  \prod_{j\neq i} \frac{x-\alpha_j}{\alpha_i-\alpha_j},
$$
on déduit que
$$\norm{P}{\rho_0}\leq \left(\frac{\abs{u-1}{p}}{\rho_0}\right)^r
  \left(\max_{0\leq i <r}\prod_{j\neq i} \frac{\rho_0} {\abs{\alpha_i-\alpha_j}{p}}\right)
  \norm{f}{\rho_0}
  =
  \frac{\abs{u-1}{p}}{\rho_0}
  \left(\max_{0\leq i <r} \prod_{j\neq i} \frac{\abs{u-1}{p}} {\abs{\alpha_i-\alpha_j}{p}}\right)
  \norm{f}{\rho_0}.
$$
Pour $i=0,\cdots,r-1$, on a
$$\prod_{j\neq i}\frac{\abs{u -1}{p}}{\abs{\alpha_i-\alpha_j}{p}}
  =\prod_{j\neq i}\frac{\abs{u -1}{p}}{\abs{u^{i-j} -1}{p}}
  =\prod_{j=0}^{i-1}p^{\ord_p(j-i)} \prod_{j=i+1}^{r-1}p^{\ord_p(j-i)}
  =(\abs{i!}{p}\abs{(r-1-i)!}{p})^{-1},
$$
donc
$$
\max_{0\leq i <r}\prod_{j\neq i}\frac{\abs{u -1}{p}}{\abs{\alpha_i-\alpha_j}{p}}
=\frac{1}{\abs{(r-1)!}{p}}.
$$
D'où,
$$\norm{P}{\rho_0}\leq
  \frac{\abs{u-1}{p}}{\rho_0}
  \frac{1}{\abs{(r-1)!}{p}}\norm{f}{\rho_0}=
  \frac{p^{\frac{1}{p-1}}\abs{u-1}{p}}{\abs{(r-1)!}{p}}\norm{f}{\rho_0}.
$$
D'autre part,
$$\norm{Q}{\rho_0}\leq \norm{f}{\rho_0}.$$
D'où
$$\norm{V}{\rho_0} \leq \max(\norm{Q}{\rho_0},\norm{P}{\rho_0})
=\max(1,p^{\frac{1}{p-1}}\frac{\abs{u-1}{p}}{\abs{(r-1)!}{p}})
\norm{f}{\rho_0}.
$$
Donc,
$$\norm{V}{1} \leq \frac{1}{\rho_0^{r-1}}\norm{V}{\rho_0}
=p^{\frac{r-1}{p-1}}\max(1,p^{\frac{1}{p-1}}\frac{\abs{u-1}{p}}{\abs{(r-1)!}{p}})
\norm{f}{\rho_0}.
$$
Montrons que $\norm{f}{\rho_0}\leq 1$.
On a $\xi_1=t(x)+\frac{x^{p-1}}{p}$ avec $t(x)\in 1+x\Zp[x]$.
Donc, pour $\rho'<\rho_0$,
$\norm{\frac{1}{\xi_1}}{\rho'}=\max_{n\geq 0}( p{\rho'}^{p-1})^n=1$.
On a donc $\norm{\frac{1}{\xi_1^{(j)}}}{\rho'}=1$
et $\norm{\frac{1}{\mlog{\Jp}{1}}}{\rho'}= 1$.
En faisant tendre $\rho'$ vers $\rho_0$, on a donc
$\norm{\frac{1}{\mlog{\Jp}{1}}}{\rho_0}\leq 1$. D'où l'inégalité
\eqref{majx}.
\end{proof}
\begin{cor}\label{aunitepres}
Lorsque $\lgg{\JJ} < p$,
  $\txin{n}{\inclus{\Jp}{\JJ}}$ est égal à $\mlog{\Jp}{n}$ à une unité près de $\Zp[[x]]$.
\end{cor}
Pour $u$ suffisament congru à 1 modulo $p$, plus précisément si
$$\abs{u-1}{p}\leq p^{-1/(p-1)}\abs{(\lgg{\JJ}-1)!}{p},$$
on a donc
$$
\norm{\txin{1}{\inclus{\Jp}{\JJ}}/\mlog{\Jp}{1}}{1} \leq p^{\lfloor\frac{\lgg{\JJ}-1}{p-1}\rfloor}
$$
Pour $\abs{u-1}{p} > p^{-1/(p-1)}\abs{(\lgg{\JJ}-1)!}{p}$,
$$\norm{\txin{1}{\inclus{\Jp}{\JJ}}/\mlog{\Jp}{1}}{1} \leq
  p^{\frac{\lgg{\JJ}-1}{p-1}}\frac{p^{-\frac{1}{p-1}}\abs{u-1}{p}}{\abs{(\lgg{\JJ}-1)!}{p}}
  \leq
  \frac{p^{\frac{\lgg{\JJ}-p-1}{p-1}}}{\abs{(\lgg{\JJ}-1)!}{p}}.
$$
Par définition, $\txin{n}{\inclus{\Jp}{\JJ}}(u^{j}\zeta-1)$ est nul pour $j \in \Jp$
et toute racine de l'unité $\zeta$ d'ordre $p^n$ et non nul pour $j \in J$
et toute racine de l'unité $\zeta$ d'ordre divisant $p^{n-1}$.
\begin{lem} Si $m > n + \frac{\log(\lfloor\frac{\lgg{\JJ}-1}{p-1} \rfloor)}{\log p}$,
$\txin{n}{\inclus{\Jp}{\JJ}}$ ne s'annule pas en $u^j\zeta-1$ pour $\zeta$ d'ordre
$p^m$.
Par exemple, si $\lgg{\JJ} < p(p-1)$ et $m\geq n$,
$\txin{n}{\inclus{\Jp}{\JJ}}$ ne s'annule pas en
d'autres $u^j\zeta_m-1$ que ceux venant de sa définition.
\end{lem}
\begin{proof}
Si $\txin{n}{\inclus{\Jp}{\JJ}}/\mlog{\Jp}{n}$ s'annule en
$u^j\zeta-1$ pour $\zeta$ d'ordre $p^m$ et $m>n$, son degré
doit être supérieur ou égal à $p^{m-1}(p-1)$, ce qui implique
que $p^{m-n} \leq \lfloor\frac{\lgg{\JJ}-1}{p-1} \rfloor$.
\end{proof}

\begin{lem}\label{lem:mu}
Pour $m > n$, on a
$$\abs{p^{\lgg{\Jp}}\txin{n}{\inclus{\Jp}{\JJ}}(u^k \zeta -1)}{p}\leq
p^{-\frac{\lgg{\Jp}}{p^{m-n}}+\tborne{\lgg{\JJ}}}$$
pour $\zeta$ racine de l'unité d'ordre $p^m$
avec égalité si $\tborne{\lgg{\JJ}}=0$.
\end{lem}
\begin{proof}
On écrit $\txin{n}{\inclus{\Jp}{\JJ}}= \mlog{\Jp}{n} V$
où $\norm{V}{1}\leq p^{\tborne{\lgg{\JJ}}}$.
Calculons la valeur absolue $p$-adique de $p^{\lgg{\Jp}}\mlog{\Jp}{n}(u^k\zeta-1)=
\prod_{j\in \Jp}\xi_n(u^{k-j} \zeta -1)
$. On a
$$\xi_n(u^{s} \zeta -1)=\frac{u^{s p^n}\zeta^{p^n}-1}{u^{s p^{n-1}}\zeta^{p^{n-1}}-1}
$$
et
\begin{equation*}
\begin{split}
\ord_p(u^{s p^n}\zeta^{p^n}-1)&= \min(\ord_p(u^{s p^n}-1), \ord_p(\zeta^{p^n}-1))\\
&= \min (\ord_p(s+n+1), \frac{1}{(p-1)p^{m-n-1}})=
\frac{1}{(p-1)p^{m-n-1}} .
\end{split}
\end{equation*}
Donc,
$ \ord_p(\xi_n(u^{s} \zeta -1))
=\frac{1}{p^{m-n}}
$
et
$\abs{p^{\lgg{\Jp}}\mlog{\Jp}{n}(u^k\zeta-1)}{p}=p^{-\frac{\lgg{\Jp}}{p^{m-n}}}$.
Si $\tborne{\lgg{\JJ}}=0$, $V$ est une unité et $\norm{V}{1}=1$.
Ce qui termine la démonstration du lemme.
\end{proof}
Soit $\lambda(Q)$ (resp. $\mu(Q)$) le $\lambda$-invariant (resp. $\mu$-invariant) $p$-adique
d'un polynôme $Q$ de $\Qp[x]$.
Rappelons que $\mu(\mlog{\JJ}{n})=-\lgg{\JJ}$. On voit facilement que
\begin{equation*}
\begin{cases}
\deg(\txin{n}{\JJ}/\mlog{\JJ}{n})&\leq p^{n-1}(\lgg{\JJ}-1)\\
\mu(\txin{n}{\JJ}/\mlog{\JJ}{n})&\geq -\tborne{\lgg{\JJ}}.
\end{cases}
\end{equation*}
\textbf{Expérimentalement}, on a
\begin{equation*}
\begin{cases}
\deg(\txin{n}{\JJ}/\mlog{\JJ}{n})&=p^{n-1}(\lgg{\JJ}-1)\\
\mu(\txin{n}{\JJ}/\mlog{\JJ}{n})&=-\tborne{\lgg{\JJ}}\\
\lambda(\txin{n}{\JJ}/\mlog{\JJ}{n})&= (p-1) p^{n-1}\lfloor \frac{\lgg{\JJ}-2}{p-1} \rfloor.
\end{cases}
\end{equation*}
\subsection{Produit infini}
On pose toujours $\rho_0=p^{-1/(p-1)}$. Pour $0 < \rho <1$, on note $n_0(\rho)$
le plus petit entier $\geq 0$ tel que $\rho^{p^{n_0(\rho)}} \leq \rho_0$.
Ainsi, $n_0(\rho_0)$ est nul.
\begin{prop}\label{convergence} Soient $r$ un entier strictement positif,
$\mu$ et $\alpha_j$ des entiers positifs pour $0 \leq j < r$.
Soit $H$ un polynôme divisible par $\prod_{j=0}^{r-1} (\omega_{n-1}^{(j)})^{\alpha_j}$.
Alors, pour tout réel $ \rho$ tel que $0 < \rho <1$,
\begin{equation*}
\norm{H}{\rho} \leq
\begin{cases}
  (p\rho_0)^{\lambda}p^{-(n-n_0(\rho))\lambda}\norm{H}{1}
 &\text{si $n\geq n_0(\rho)$}\\
 (\rho^{p^n})^{\lambda}\norm{H}{1} &\text{si $n\leq n_0(\rho)$}
\end{cases}
\end{equation*}
avec
$\lambda = \sum_{j=0}^{r-1} \alpha_j$.
\end{prop}
\begin{proof}
On a $$H = \prod_{j=0}^{r-1} (u^{-jp^{n-1}}(1+x)^{p^{n-1}}-1)^{\alpha_j} \times W$$
avec $W \in \Qp[x]$. Par
le théorème de la division euclidienne par un polynôme unitaire de $\Zp[x]$,
$\norm{W}{1}$ est inférieur à $\norm{H}{1}$. Donc,
\begin{equation*}
\norm{ H}{\rho} \leq
\prod_{j=0}^{r-1} \norm{u^{-jp^{n-1}}(1+x)^{p^{n-1}}-1}{\rho}^{\alpha_j}\norm{H}{1}.
\end{equation*}
Posons $\varphi(P)=P((1+x)^p -1)$ pour $P$ polynôme en $x$.
On a
\begin{equation*}
\norm{\varphi(P)}{\rho} \leq \begin{cases}
\norm{P}{\rho/p} &\text{ si $\rho \leq p^{-1/(p-1)}$}\\
\norm{P}{\rho^p} &\text{ si $\rho \geq p^{-1/(p-1)}$}.
\end{cases}
\end{equation*}
Supposons d'abord $n \geq n_0(\rho)$. Alors,
\begin{equation*}
\begin{split}
\norm{\varphi^{n-1}
(u^{-jp^{n-1}}(1+x)-1)}{\rho}&\leq
  \norm{u^{-jp^{n-1}}(1+x)-1}{\rho^{p^{n_0(\rho)}}/p^{n-1-n_0(\rho)}} \\
  &\leq
  \norm{u^{-jp^{n-1}}(1+x)-1}{\rho_0/p^{n-1-n_0(\rho)}}
  \\
  &\leq \max(\abs{u^{-jp^{n-1}}-1}{p},p^{-n+1+n_0(\rho)}\rho_0).
\end{split}
\end{equation*}
Si $j\neq 0$,
\begin{equation*}
\begin{split}
\norm{\varphi^{n-1}
(u^{-jp^{n-1}}(1+x)-1)}{\rho}
&\leq \max(\abs{u^{p^{n-1+ \ord_p(j)}}-1}{p},p^{-n+1+n_0(\rho)}\rho_0)\\
&\leq p^{-(n-1)}
\max(\abs{u-1}{p}p^{-\ord_p(j)},p^{n_0(\rho)}\rho_0)\\
&\leq p^{-(n-1)}
\max(\abs{u-1}{p},p^{n_0(\rho)}\rho_0) .
\end{split}
\end{equation*}
Donc, pour $n \geq n_0(\rho)$,
\begin{equation*}
\begin{split}
\norm{ H }{\rho}
&\leq p^{-(n-1)\lambda}\max(\abs{u-1}{p},p^{n_0(\rho)}\rho_0)^{\lambda}\norm{H}{1}\\
&\leq (p\rho_0)^{\lambda}
  p^{-(n-n_0(\rho))\lambda}\norm{H}{1}.
\end{split}
\end{equation*}
Pour $0 < n < n_0(\rho)$,
\begin{equation*}
\begin{split}
\norm{\varphi^{n-1}(u^{-jp^{n-1}}(1+x)-1}{\rho}
&\leq \norm{u^{-jp^{n-1}}(1+x)-1}{\rho^{p^n}}\\
&\leq\max(\abs{u^{p^{n-1 + \ord_p(j)}}-1}{p},\rho^{p^n}) \leq \rho^{p^n}
\end{split}
\end{equation*}
car $\abs{u^{p^{n-1 + \ord_p(j)}}-1}{p} \leq p^{-1} < p^{-1/(p-1)} < \rho^{p^n}<1$.
Donc, pour $0 < n < n_0(\rho)$,
\begin{equation*}
\norm{ H }{\rho}\leq (\rho^{p^n})^\lambda \norm{H}{1}
\leq \norm{H}{1}.
\end{equation*}
\end{proof}
\begin{defn}
Soient $\lambda >0$, $\mu \geq 0$.
On dit qu'un produit infini $\prod_n G_n$ est \textsl{de type
($\lambda$, $\mu$)} si
\begin{enumerate}
  \item $ \norm{G_n}{1} \leq p^\mu\ ;$
  \item il existe une constante $\nu$ telle que pour tout réel $ \rho$, $0 < \rho <1$,
  \begin{equation}\label{type}
\norm{G_n - 1}{\rho} \leq p^\nu p^{-\lambda(n -n_0(\rho)) }
\quad \text{si $n\geq n_0(\rho)$}.
\end{equation}
\end{enumerate}
\end{defn}
\begin{prop}
\label{produit}
Un produit $ \prod_n G_n$ de type ($\lambda$, $\mu$) avec $\lambda >0$, $\mu \geq 0$
définit un élément $G$ de
$\cH$ d'ordre inférieur ou égal à $\mu$.
\end{prop}
La factorisation $G=\prod_n G_n$ est dite de
\textsl{vitesse de convergence} $\geq \lambda$.

\begin{proof}
On déduit de l'équation \eqref{type} que pour $0<\rho<1$,
$\norm{G_n-1}{\rho}$
tend vers 0 lorsque $n \to \infty$,
ce qui prouve la convergence de $\prod_n G_n=\prod_{n=1}^\infty\left( 1 + (G_n-1) \right)$.
On définit ainsi un élément $G$ de $\cH$.
Rappelons que l'ordre de $G$ est le minimum des $h$ tels que
$\norm{p^{mh}G}{\rho^{1/p^m}}$ est borné par rapport à $m$ pour un $\rho<1$.
Prenons $\rho=\rho_0=p^{-\frac{1}{p-1}}$. Comme $n_0(\rho_0^{1/p^m}) = m$,
\begin{equation*}
\norm{G_n - 1}{\rho_0^{1/p^m}} \leq
\begin{cases}
 p^{\nu}p^{-\lambda(n-m)} &\text{si $n\geq m$}\\
 p^\mu &\text{si $n\leq m$}.
\end{cases}
\end{equation*}
On en déduit
que si $n\geq m + \frac{\nu}{\lambda}$,
$\norm{G_n}{\rho_0^{1/p^m}}\leq 1$.
Donc
\begin{equation*}
\begin{split}
\prod_{n=1}^\infty \norm{G_n}{\rho_0^{1/p^m}} &\leq
\prod_{n=1}^{m+\lfloor\frac{\nu}{\lambda}\rfloor} \norm{G_n}{\rho_0^{1/p^m}}
\leq p^{\lfloor\frac{\nu}{\lambda}\rfloor\mu}p^{m\mu} \, \\
\norm{p^{m\mu}G}{\rho_0^{1/p^m}} &\leq p^{\lfloor\frac{\nu}{\lambda}\rfloor\mu}
\end{split}
\end{equation*}
et $G$ est d'ordre inférieur ou égal à $\mu$.
\end{proof}
\begin{prop}
\label{prop:majoration}
Pour tout réel $ \rho$ tel que $0 < \rho <1$, on a
\begin{equation*}
\norm{\txin{n}{\inclus{\Jp}{\JJ}}-1}{\rho}
\leq \begin{cases}
p^{\frac{p-2}{p-1}\lgg{\JJ}+\tborne{\lgg{\JJ}} + \lgg{\Jp}}p^{-(n-n_0(\rho))\lgg{\JJ}} &\text{ si $n\geq n_0(\rho)$}\\
p^{\tborne{\lgg{\JJ}} + \lgg{\Jp}} &\text{ si $n\leq n_0(\rho)$}
\end{cases}
\end{equation*}
pour $n > \ord_p((\lgg{\JJ}-1)!) - (\ord_p(u-1)-1)$.
\end{prop}
\begin{proof}
Comme $\norm{\txin{n}{\inclus{\Jp}{\JJ}}}{1} \leq p^{\tborne{\lgg{\JJ}}+\lgg{\Jp}}$, il en est de même de
$\norm{\txin{n}{\inclus{\Jp}{\JJ}}-1}{1}$. D'autre part,
$\txin{n}{\inclus{\Jp}{\JJ}}-1$ est divisible par
$\prod_{j\in J}\omega_{n-1}^{(j)} $
car
$\prod_{j \in \Jp} (\frac{\xi_{n}^{(j)}}{p}-1)
$
est un multiple de $\prod_{j\in \Jp}\omega_{n-1}^{(j)} $.
On applique la proposition \ref{convergence}
à $H=\txin{n}{\inclus{\Jp}{\JJ}}-1$
avec $\norm{H}{1}\leq p^{\tborne{\lgg{\JJ}}+\lgg{\Jp}}$ et $\lambda=\lgg{\JJ}$.
\end{proof}
\begin{prop}\label{xin}
Soient  $\JJ$ un intervalle fini de $\ZZ$
et  $\Jp$ un sous-intervalle de  $\JJ$.
Soit $N$ un entier positif.

Le produit infini $\mxiinf{N}{\inclus{\Jp}{\JJ}}=\prod_{m=N}^\infty \txin{m}{\inclus{\Jp}{\JJ}}$\index{$\mxiinf{N}{\inclus{\Jp}{\JJ}}$}
est de vitesse de convergence supérieure ou égale à $\lgg{\JJ}$,
définit une fonction d'ordre compris entre $\lgg{\Jp}$ et $\lgg{\Jp} + \tborne{\lgg{\JJ}}$
telle que pour
$\zeta$ racine de l'unité d'ordre $p^{n}$,
\begin{equation*}
\mxiinf{N}{\inclus{\Jp}{\JJ}} (u^{j}\zeta -1)=\begin{cases}
  0 &\text{ si $n \geq N$ et $j \in \Jp$}\\
  1 &\text{ si $n < N$ et $j\in \JJ$}.
\end{cases}
\end{equation*}
et est divisible par $\mloginf{\Jp}{N}=\prod_{m=N}^{\infty}\mlog{\Jp}{m}$ dans $\cH$.
En particulier, lorsque $\lgg{\JJ} < p$, les deux fonctions $\mxiinf{N}{\inclus{\Jp}{\JJ}}$ et
$\mloginf{\Jp}{N}$ ont même ordre $\lgg{\Jp}$ et le quotient
$\mxiinf{N}{\inclus{\Jp}{\JJ}}/\mloginf{\Jp}{N}$\index{$\mxiinf{N}{\inclus{\Jp}{\JJ}}$}
est une unité de $\Zp[[x]]$.
\end{prop}
\
\begin{proof}
On applique la proposition \ref{produit} avec les majorations
de la proposition \ref{prop:majoration} pour obtenir la majoration de l'ordre.
Soit $\zeta$ une racine de l'unité d'ordre $p^n$.
Pour  $\JJ \neq \emptyset$ et $j\in J$, on a
$\mxiinf{N}{\inclus{\Jp}{\JJ}} (u^{j}\zeta-1)= 0$ si $n\geq N$ et $1$ si $n<N$
pour $j \in \Jp$.
La fonction $\mxiinf{N}{\inclus{\Jp}{\JJ}}$ est donc divisible dans $\cH$ par
$\mloginf{\Jp}{N}$.
L'ordre du logarithme $p$-adique étant égal à 1,
on en déduit la minoration de l'ordre de $\mxiinf{N}{\inclus{\Jp}{\JJ}}$.
Le quotient est une unité grâce au corollaire \ref{aunitepres}.
\end{proof}
On pose $\mxiinf{N}{\emptyset}= \mxiinf{N}{\inclus{\emptyset}{\JJ}}=1$
et
$\mxiinf{N}{\JJ}=\mxiinf{N}{\inclus{\JJ}{\JJ}}$.
\begin{rem}
On espère en général la non-nullité de
$\mxiinf{N}{\inclus{\Jp}{\JJ}} (u^{j}\zeta -1)$
si $\zeta$ est une racine de l'unité d'ordre $p^n$ avec $n \geq N$ et $j \in \JJ-\Jp$
(démontré pour $\lgg{\JJ} < p$).
\end{rem}

\section{$\Lambda$-modules associés à un $\varphi$-module filtré}
\subsection{Polygones associés à un $\varphi$-module filtré}

Nous rassemblons ici quelques définitions et notions classiques
(\cite{fontaine}, \cite{katz}, \dots).

Le \textsl{polygone associé à la suite croissante de rationnels
$(a_1,\cdots, a_d)$} est le polygone dont les sommets sont $(0,0)$ et
les points $(i, \sum_{j=1}^i a_j)$ pour $i=1,\cdots, d$.
Autrement dit, c'est le polygone partant de $(0,0)$ et dont les pentes
successives sont les $a_i$ auquel on rajoute les demi-droites verticales
entre $(0,0)$ et $(0,\infty)$ et entre $(d,\sum_{j=1}^d a_i)$ et
$(d,\infty)$. La suite des pentes $(a_i)$ étant croissante,
le polygone est convexe et tourne sa concavité vers $+\infty$.
\footnote{On utilisera aussi le terme de polygone pour l'enveloppe convexe.}
Le polygone associé à la suite croissante $(a_1,\cdots, a_d)$
est au dessous du polygone associé à la suite croissante $(b_1,\cdots, b_d)$
si et seulement si $a_1 + \cdots + a_j \leq b_1 + \cdots + b_j$
pour tout $j=1, \cdots, d$.

Soit $A$ un anneau intègre possédant une théorie des diviseurs élémentaires.
Soit $M$ un $A$-module de type fini sans torsion
et $f$ un $A$-morphisme injectif de $M$ dans $ M$.
On appelle \textsl{suite des diviseurs élémentaires} (invariants de
Smith\footnote{en lien avec la
notion de matrice de Smith.}) de $f$ sur $M$
la suite des diviseurs élémentaires
$(a_1,\cdots, a_d)$ de $f(M) \subset M$ normalisée de manière à ce que
$$(a_d) \subset \cdots \subset (a_1)$$
(chaque $a_i$ est défini à une unité près).
La définition s'étend de manière naturelle lorsque $f$ est un homomorphisme
de $M$ dans $a^{-1} M$ pour un $a\in A$,
la suite des diviseurs élémentaires de $f$ sur $M$ est le produit de $a^{-1}$ par la suite des diviseurs élémentaires de $af $ sur $M$.
Une base $\cB= (v_1,\cdots,v_d)$ de $M$ est dite \textsl{adaptée} à $f$ si
 pour $(a_1, \cdots, a_d)$ la suite des diviseurs élémentaires,
 $$\sum_i A a_i^{-1}f(v_i) = M$$
ou encore en notant $M^{(i)}$ le $A$-module engendré par $(v_i, \cdots, v_d)$, si
 $$\sum_i A a_i^{-1}f(M^{(i)}) = M.$$
La matrice de $f$ dans la base $\cB$ est de la forme $V F$
avec $V \in GL_d(A)$ et $F$ la matrice diagonale de diagonale
$a_1, \cdots, a_d$.
Lorsque $A=\Zp$, il existe une suite croissante d'entiers
$s_1(M)$, $\cdots$, $s_d(M)$ telle que $\sum_i \Zp p^{-s_i(M)}f(v_i) =M$.
Les $p^{s_i(M)}$ sont les diviseurs élémentaires de $f(M)$ dans $M$.
La suite des $M^{(i)}$ forme une filtration décroissante de $M$
associée à $f$ et à $\cB$.
\begin{defn}[Polygone de Smith sur $\Zp$]
On appelle \textsl{polygone de Smith} de $(M,f)$ le polygone associé
à la suite des valuations $p$-adiques des diviseurs élémentaires de $f$ sur $M$.
\end{defn}
\begin{defn}Soit $\DD$ un $\varphi$-module, c'est-à-dire un $\EEE$-espace vectoriel
muni d'un endomorphisme bijectif.
Le \textsl{polygone de Newton associé à $(\DD,\varphi)$} est le polygone de
Newton du polynôme caractéristique $\det_{\DD}(1 - x \varphi)$.
Si $\det_{\DD}(1 - x \varphi)=1+a_1 x+ \cdots + a_d x^d$,
c'est le bord de l'enveloppe complexe des points $(i,\ord_p(a_i))$
pour $i=0,\cdots,d$ dans le plan.
C'est aussi le polygone associé à la suite croissante des valuations des valeurs propres de
$\varphi$.
\end{defn}
\begin{prop}[\cite{katz}, Katz, Mazur] Soient $\DD$ un $\varphi$-module
et $M$ un sous-$\Zp$-module de $\DD$.
Le polygone de Smith de $\varphi$ sur $M$ et le polygone de Newton de
$(\DD,\varphi)$ ont mêmes extrémités et
le polygone de Smith de $\varphi$ sur $M$ est au dessous du polygone de Newton
de $\varphi$.
\end{prop}
Le polygone que nous appelons ici de Smith est appelé polygone de Hodge dans
\cite{katz}.
Il a d'ailleurs vocation à être égal au polygone dit de Hodge-Tate.
\begin{defn}
Soit $\DD$ un \textsl{module filtré} de dimension $d$, c'est-à-dire
un $\EEE$-espace vectoriel muni d'une filtration $\Fil^\bullet \DD$
de $\EEE$-espaces vectoriels, décroissante, exhaustive et séparée.
Soit $(t_{\HT,1}, \cdots, t_{\HT,d})$ la suite croissante des entiers $i$ tels que
$\Fil^{i}\DD/\Fil^{i+1}\DD$ est non nul,
comptés avec une multiplicité égale à $\dim_{\EEE}\Fil^{i}\DD/\Fil^{i+1}\DD$
(\textsl{poids de Hodge-Tate}).
\end{defn}
On a donc
$\Fil^{t_{\HT,1}} \DD=\DD$, $\Fil^{t_{\HT,1}+1} \DD\neq \DD$, $\Fil^{t_{\HT,d}}\DD \neq 0$,
$\Fil^{t_{\HT,d}+1}\DD = 0$.
Si $v\in \DD$, soit $t_H(v)$ le plus grand entier $j$ tel que $v \in \Fil^j \DD$.
\begin{defn}
Le \textsl{polygone de Hodge-Tate d'un module filtré} est le polygone associé
à la suite croissante de ses poids de Hodge-Tate.
\end{defn}
\begin{defn}
Une \textsl{base adaptée à la filtration} est une base $(v_1,\cdots, v_d)$
telle que $v_i\in \Fil^{t_{\HT,i}} \DD$ pour tout $i$. Ainsi,
$(v_i, \cdots, v_d)$ est contenu dans $\Fil^{t_{\HT,i}}\DD$.
Si les poids de Hodge-Tate sont distincts,
$(v_i, \cdots, v_d)$ est une base de $\Fil^{t_{\HT,i}}\DD$.
\end{defn}

Soit $(\Delta^i)$ une \textsl{graduation de la filtration} $\Fil^\bullet $,
c'est-à-dire
une famille de sous-espaces $\Delta^i$ de $\Fil^i \DD$ tels que
la projection $\Delta^{i}\to \Fil^{i}\DD/\Fil^{i+1}\DD$ est un isomorphisme.
L'ensemble des $i$ tels que
$\Delta^i\neq 0$ est l'ensemble des poids de Hodge-Tate. On a
$\Fil^i \DD=\oplus_{j\geq i} \Delta^j$
ou
$\Fil^{t_{\HT,i}} \DD=\oplus_{j\geq t_{\HT,i}} \Delta^{j}$.
Une base $\cB=(v_1,\cdots,v_d)$ adaptée
à la filtration permet de définir une telle graduation
$\Delta^\bullet$: si $\Delta^i \neq 0$,
$\Delta^i$ est alors le sous-espace engendré par
les $v_j$ tels que $t_H(v_j)=i$.

\begin{defn}
Soit $(\DD, \Fil^{\bullet})$ un module filtré et
$(\Delta^i)$ une graduation de la filtration.
On définit l'\textsl{homomorphisme de pentes de Hodge-Tate} $t_{\HT,\Delta}$ par
$t_{\HT,\Delta}(v)=p^{i} v$ pour $v \in \Delta^i$.
Si $\cB=(v_1,\cdots,v_d)$ est une base adaptée à la filtration,
on note $t_{\HT,\cB}$ l'endomorphisme défini par
$t_{\HT,\cB}(v_i)=p^{t_{\HT}(v_i)} v_i$. La définition est compatible
avec celle de $t_{\HT,\Delta}$ pour $\Delta$ associée à $\cB$.
\end{defn}
\begin{defn}
 Une base $\cB$ est dite \textsl{adaptée au $\varphi$-module filtré
$(\DD, \varphi, \Fil^\bullet \DD)$} si
$\cB$ est une base adaptée à la fois à $\varphi$ et à la filtration $\Fil^\bullet \DD$.
\end{defn}
Il est équivalent de dire que la filtration associée à $\varphi$
et à $\cB$ coincide avec la filtration $\Fil^\bullet$.
\begin{defn}[Fontaine, \cite{fontaine} par exemple]
\label{def:faibleadm}
Un réseau $M$ d'un $\varphi$-module filtré $(\DD,\varphi,\Fil^\bullet)$ est dit
\textsl{fortement divisible} s'il existe une base de $M$ adaptée
au $\varphi$-module filtré $(\DD, \varphi, \Fil^\bullet \DD)$
et si le polygone de Smith de $\varphi$ sur $M$
est égal au polygone de Hodge-Tate de $(\DD, \Fil^\bullet \DD)$.
Si $\DD$ admet un tel réseau, $\cD$ est dit \textsl{faiblement admissible}.
\end{defn}
\begin{lem}
Un réseau $M$ de $\DD$ est fortement divisible si et seulement si $M$ admet
une base $\cB$ adaptée au $\varphi$-module filtré $(\DD, \varphi, \Fil^\bullet \DD)$
telle que $\varphi \circ t_{\HT,\cB}^{-1} (M)=M$.
\end{lem}
\begin{proof}
Prenons une base $\cB$ de $M$ adaptée à la filtration et à $\varphi$.
L'égalité des extrémités des polygones de Hodge-Tate et de Smith
signifie que le déterminant de $\varphi$
est de valuation égale à $\sum_{v\in \cB} t_H(v)$
(propriété indépendante de la base ou du réseau choisi).
Le déterminant de $\varphi\circ t_{\HT,\cB}^{-1}$ est donc une unité.

Supposons l'égalité des polygones de Smith de $(M,\varphi)$
et de Hodge-Tate de $\DD$:
on a $\varphi (v) \in p^{t_H(v)} M$ pour $v\in \cB$. Donc
$\varphi\circ t_{\HT,\cB}^{-1}(M)\subset M$.
Il y a donc égalité.
La réciproque est claire par définition des diviseurs élémentaires de $\varphi$.
\end{proof}

\begin{thm}[Laffaille \cite{FL}]
Le $\varphi$-module filtré $(\DD, \varphi, \Fil^\bullet \DD)$
est faiblement admissible si et seulement si le critère numérique suivant est vérifié :
\begin{enumerate}
\item le polygone de Hodge-Tate de $(\DD,\Fil^\bullet)$ et le polygone de Newton de
$(\DD, \varphi)$ ont mêmes extrémités;
\item pour tout sous-module $\DD'$ de $\DD$ stable par $\varphi$,
le polygone de Hodge-Tate de $(\DD', \DD' \cap \Fil^\bullet \DD )$ est en dessous du polygone
de Newton de $(\DD', \varphi_{| \DD'})$.
\end{enumerate}
\end{thm}
Les définitions de raffinement ne sont pas les mêmes dans \cite{Mazur}
et dans \cite{bell-chen}.

\begin{defn}[\cite{Mazur}]
Un \textsl{raffinement d'un $\varphi$-module $\DD$} est un
drapeau complet $\cF_\bullet=(\cF_i)$ de $\cD$ stable par $\varphi$.
Une \textsl{base adaptée au raffinement $\cF$} est une base $(v_1, \cdots, v_d)$ de $\cD$ telle que
$\cF_i = ( v_{1}, \cdots, v_i )$.
\end{defn}
\begin{defn}
Un \textsl{raffinement d'un $\varphi$-module filtré $\DD$} est
la donnée d'un raffinement $\cF_\bullet$ du $\varphi$-module $\DD$ tel que
$\cF_{i} \cap \Fil^{t_{H,i}} \DD \neq \{0\}$ et
$\cF_{i} \cap \Fil^{t_{H,i}+1} \DD = \{0\}$
pour tout $1 \leq i \leq d$.
\end{defn}
Si de plus les poids de Hodge-Tate sont distincts, on a
$\DD=\cF_{i} \oplus \Fil^{t_{H,i}+1} \DD$.
On parle aussi de raffinement non critique de $\cD$ (\cite{bell-chen}).
\begin{defn}
Une base $\cB$ est dite \textsl{adaptée à un raffinement $\cF_\bullet$}
du $\varphi$-module filtré $(\DD, \varphi, \Fil^\bullet \DD)$
si $\cB$ est adaptée à la filtration $\Fil^\bullet \DD$ et au drapeau
$\cF_\bullet$.
\end{defn}
En général, une telle base $\cB$ n'est pas une base adaptée
au $\varphi$-module filtré $(\DD, \varphi, \Fil^\bullet \DD)$.
\begin{rem}
Supposons les poids de Hodge-Tate tous distincts de $\cD$ et $\varphi$ diagonalisable.
Si $\cB=( v_1,\cdots,v_d )$ est adaptée au raffinement
$\cF_\bullet$,
$(v_i, \cdots, v_d)$ est une base de $\Fil^{t_{H,i}} \DD$
et $(v_1,\cdots,v_i)$ est une base de $\cF_i$.
La matrice de $t_{\HT,\cB}$ dans la base $\cB$ est la matrice diagonale
$\smallmat{p^{t_{H,1}}\\
&\ddots\\
&& p^{t_{H,d}}}$. Il existe une base $\cB_0= (e_1,\cdots,e_d)$
formée de vecteurs propres
telle que $( e_1,\cdots,e_i)$ est une base de
$\cF_i$.
La matrice $P$ des vecteurs de $\cB$ dans la base $\cB_0$ est triangulaire supérieure
et peut être supposée unipotente :
$P =((\lambda_{i,j}))$ avec $\lambda_{i,j}=0 $ si
$i>j$ et $\lambda_{i,i}=1$. Si $\varphi_{\cB_0}$ est la matrice diagonale de $\varphi$
dans la base $\cB_0$, le réseau engendré par $\cB$ est fortement divisible si la matrice unipotente
$P\varphi_{\cB_0}P^{-1}\varphi_{\cB_0}^{-1}$ appartient à $\Gl_d(\Zp)$.
Cela se traduit par des conditions explicites
sur la valuation $p$-adique des $\lambda_{i,j}$.
\end{rem}
\begin{defn}
On note $\Qp[i]$ le $\varphi$-module filtré de Tate
défini comme le $\Qp$-espace vectoriel $\Qp$
muni de la filtration $\Fil^i \Qp[i]=\Qp$, $\Fil^{i+1} \Qp[i]=0$
et de l'endomorphisme $\varphi$ donné par la multiplication par $p^i$.
On note $\DD[i]=\DD \otimes \Qp[i]$ le twist de $\DD$ par $\Qp[i]$.
\end{defn}

\subsection{Définition des modules logarithmiques.}
On pose $K_n=\Qp(\mu_{p^n})$.
\begin{defn}
\begin{enumerate}
\item Pour $N\geq 0$, on note $\LLL{}{N}{\DD}\index{$\LLL{}{N}{\DD}$}$ le sous-module
de $\cH_\oEE \otimes \DD$ formé des éléments $G$ tels que
$$\phiun^{\nm}\Gg(u^{j}\zeta-1)
\in K_n \otimes\ot \Fil^{j} \DD$$ pour tous $j\leq t_{\HT,d}$ et $\zeta$ racine de l'unité
d'ordre $p^n$ avec $n \geq N$.
\item On note $\LLL{}{eul}{\DD}$ le sous-module
de $\LLL{}{1}{\DD}$ formé des éléments $\Gg$ tels que
$$(1-p^{-1+j}\varphi^{-1})\Gg(u^{j} - 1)
\in (1-p^{-j}\varphi)\Fil^{j} \DD$$
pour $j\leq t_{\HT,d}$.
\item
Pour $N<0$, on note $\LLL{\JJ}{N}{\DD}$ le sous-module
de $\LLL{}{0}{\DD}$ formé des éléments $\Gg$ vérifiant
$\Gg(u^{j}\zeta-1)= 0 $ pour tous $j \in ]t_{\HT,1},t_{\HT,d}]$ et $\zeta$ d'ordre divisant
$p^{\abs{N}{}-1}$.
\end{enumerate}
\end{defn}
Dans la définition de $\LLL{}{N}{\DD}$ pour $N\geq 0$, la condition
pour $j \leq t_{\HT,1}$ est automatique car $\Fil^{j}\DD$ est alors
égal à $\DD$.
Si $\DD$ a un seul poids de Hodge-Tate, $\LLL{}{N}{\DD}$
est $\cH\oEE \otimes \DD$
pour tout $N$.
\begin{defn}
\begin{enumerate}
\item
Soit $\DD$ un $\varphi$-module.
Un élément $\Gg$ de $\cH \otimes \DD$ est dit d'\textsl{ordre fini} s'il existe
un réel $t$ tel que la suite $(\norm{p^{t n}\phiun^{-n}\Gg}{\rho^{1/p^n}})_n$ est bornée
pour un $\rho <1$.
On définit alors sa \textsl{pente} comme le minimum de tels $t$.
\item
Soit $\DD$ un $\varphi$-module filtré.
Un élément $\Gg$ de $\cH \otimes \DD$ est dit de \textsl{pente finie}
s'il existe
un réel $t$ tel que la suite
$(\norm{p^{(t+t_{\HT,d}) n}\phiun^{-n}\Gg}{\rho^{1/p^n}})_n$ est bornée
pour un $\rho <1$. On définit alors sa \textsl{pente}
comme le minimum de tels $t$.
\end{enumerate}
\end{defn}
La notion de pente est invariante par twist par $\Qp[i]$
dans le cas des $\varphi$-modules filtrés, car
$t_{\HT,d}(\DD[i])=t_{\HT,d}(\DD)+i, \;\varphi_{\DD[i]}=p^i\varphi_{\DD}$.
Un $\varphi$-module peut être considéré comme un $\varphi$-module filtré de filtration
triviale $\Fil^0 \DD=\DD$, $\Fil^1 \DD=0$. Les deux notions coincident alors
et sont des cas particuliers de la notion de pente que l'on
trouve dans \cite{kedlaya08}, mais nous gardons le terme
d'ordre (ou de $\varphi$-ordre) utilisé dans \cite{bpr-debut} afin de distinguer.
\begin{defn}Soit $\DD$ un $\varphi$-module filtré.
Si $M$ est un sous-$\cH$-module de $\cH \otimes \DD$,
on définit $\Er{M}{t}$ comme le sous-$\Lambda\ot$-module engendré par
les éléments de $M$ de pente inférieure ou égale à $t$.
Un sous-$\cH$-module $M$ de $\cH \otimes \DD$ est dit
de pente $t$ si $t$ est le plus petit réel tel que $\cH \Er{M}{t} = M$.
\end{defn}
\subsection{Construction d'éléments d'ordre fini de $\LLL{}{}{\DD}$}
Fixons une base $\cB$ de $\DD$ adaptée à la filtration
de $\DD$. Soient $\Delta_\cB$ la graduation associée sur $\Zp$
et
$M$ le réseau de $\DD$ engendré par $\cB$.
On utilise comme norme d'un endomorphisme le maximum des normes des coefficients
de sa matrice dans la base $\cB$:
on a $\norm{fg}{\cB}\leq \norm{f}{\cB}\norm{g}{\cB}$.

Posons $\newtH$\index{$\newtH$} le plus petit entier supérieur ou égal à $t_{\HT,d}$
et tel que $\newtH-t_{\HT,1}>t_{\Ne,d}-t_{\Ne,1}$.
Dans les cas particuliers étudiés ensuite, on a
$\newtH=t_{\HT,d}$ ou $t_{HT,d}+1$.
Dans la suite,  $\JJ$ est un intervalle de $\ZZ$ contenant $\IIIprim{t_{\HT,1}}$.
En pratique, on prendra  $\JJ$ de la forme $\II{t_{\HT,1}}{h}$ avec $h\geq \newtH$.
Pour $n\geq 0$, on note
$\TH{\Delta_\cB}{n}{\JJ}$ l'endomorphisme de $\cH\oEE\otimes_\oEE \DD$ défini
pour $v\in \Delta_\cB^i$ par
\begin{equation*}
\TH{\Delta_\cB}{n}{\JJ}(v)=\txin{n}{\inclus{\III{i}}{\JJ}}{} \cdot v
\end{equation*}\index{$\TH{\Delta_\cB}{n}{\JJ}$}
(lorsque $\Delta_\cB^i$ est non nul, $i$ est un des $t_{\HT,j}$).
\begin{thm}\label{zndiv} Soit $\DD$ un $\varphi$-module filtré. Pour $\JJ$
contenant $\IIIprim{t_{\HT,1}}$, définissons par récurrence
la suite d'applications $\cH\oEE$-linéaires de $\cH\oEE\otimes\oEE \DD$ à valeurs
dans $\cH\oEE\otimes\oEE \DD$
\begin{equation*}
\ZZZ{\JJ}{\cB}{n}{N}=
 \phiun^{\np} \TH{\Delta_\cB}{n}{\JJ}
\phiun^{\nm} \ZZZ{\JJ}{\cB}{n-1}{N} \text{ pour $n > N$}
\end{equation*}\index{$\ZZZ{\JJ}{\cB}{n}{N}$}
et
$\ZZZ{\JJ}{\cB}{N}{N}=\id$.
On a pour $n \geq N$
\begin{equation}
\label{zndiv1}
\ZZZ{\JJ}{\cB}{n+1}{N} \equiv \ZZZ{\JJ}{\cB}{n}{N} \bmod \bomega{n}{\JJ}.
\end{equation}
\label{zndiv2}
  Pour $v\in \DD$, la suite des $\ZZZ{\JJ}{\cB}{n}{N}(v)$ converge
  dans $\cH\oEE\otimes\oEE \DD$.
\label{zndiv4}Le déterminant de la limite $\Zinftyu{\JJ}{\cB}{N}$ est égal à
$\prod_{k=1}^{d} \mxiinf{N}{\inclus{\III{t_{\HT,k}}}{\JJ}}$
à une unité près de $\cH$.
\end{thm}
On aurait pu indexer les objets introduits par $h$ plutôt que par $\JJ$
 avec $h\geq \newtH$ ($J=\II{t_{\HT,1}}{h}$, voir \cite{bpr-debut}).
\begin{proof}[Démonstration du théorème \ref{zndiv}]
Montrons l'équation \eqref{zndiv1}. On a pour $n>N$
\begin{equation}\label{zn:convergence}
\ZZZ{\JJ}{\cB}{n}{N} -\ZZZ{\JJ}{\cB}{n-1}{N}
  =\phiun^{\np} (\TH{\Delta_\cB}{n}{\JJ} - \id)
  \phiun^{\nm} \ZZZ{\JJ}{\cB}{n-1}{N}.
\end{equation}
Sur $\Delta_\cB^i$ non réduit à $\{0\}$ (on a alors $i=t_{\HT,j}$ pour un $j$ compris
entre 1 et $d$), $\TH{\Delta_\cB}{n}{\JJ} - \id$
est la multiplication par
$\txin{n}{\inclus{\III{i}}{\JJ}}-1$ et
est divisible par $\bomega{n-1}{\JJ}$.
Donc,
$$\ZZZ{\JJ}{\cB}{n}{N} \equiv \ZZZ{\JJ}{\cB}{n-1}{N} \bmod \bomega{n-1}{\JJ}\ ,$$
d'où l'équation \eqref{zndiv1}.
Pour $0<\rho<1$, on a
\begin{equation*}
\norm{\TH{\Delta_\cB}{n}{\JJ} - \id}{\cB,\rho} =
  \max_i\norm{\txin{n}{\inclus{\III{t_{\HT,i}}}{\JJ}}-1}{\rho}.
\end{equation*}
En appliquant la proposition \ref{prop:majoration}, pour $n \geq n_0(\rho)$,
\begin{equation}
\label{majoration}
\norm{\TH{\Delta_\cB}{n}{\JJ} - \id}{\cB,\rho} \leq
p^{\frac{p-2}{p-1}\lgg{\JJ}+\tborne{\lgg{\JJ}}+t_{\HT,d}-t_{\HT,1}} p^{-(n-n_0(\rho))\lgg{\JJ}}.
\end{equation}
Donc, pour $n\geq n_0(\rho)$,
\begin{equation*}
\begin{split}
\norm{\ZZZ{\JJ}{\cB}{n}{N}-\ZZZ{\JJ}{\cB}{n-1}{N}}{\rho}
&\leq \norm{\varphi^{\np}}{\cB}\norm{\varphi^{\nm}}{\cB}
\norm{\TH{\Delta_\cB}{n}{\JJ}-\id}{\rho}\norm{\ZZZ{\JJ}{\cB}{n-1}{N}}{\rho}
\\
&\leq \norm{\varphi^{\np}}{\cB}\norm{\varphi^{\nm}}{\cB}
p^{\frac{p-2}{p-1}\lgg{\JJ}+\tborne{\lgg{\JJ}}+t_{\HT,d}-t_{\HT,1}} p^{-(n-n_0(\rho))\lgg{\JJ}}\norm{\ZZZ{\JJ}{\cB}{n-1}{N}}{\rho}.
\end{split}
\end{equation*}
On a
$$\lim_n \frac{\ord_p(\norm{\varphi^{\np}}{\cB})}{n}=-t_{\Ne,1},\quad
\lim_n\frac{\ord_p(\norm{\varphi^{\nm}}{\cB})}{n}=t_{\Ne,d}.$$
On en déduit que pour $n$ assez grand,
\begin{equation*}
\norm{\ZZZ{\JJ}{\cB}{n}{N}-\ZZZ{\JJ}{\cB}{n-1}{N}}{\rho}
\leq C(\rho) p^{-(n-n_0(\rho))(\lgg{\JJ}-t_{\Ne,d}+t_{\Ne,1})}
\norm{\ZZZ{\JJ}{\cB}{n-1}{N}}{\rho}.
\end{equation*}
Par hypothèse, $\lgg{\JJ}\geq \newtH-t_{\HT,1}>t_{\Ne,d}-t_{\Ne,1}$. Donc, pour $n$ assez grand, on a
$$\norm{\ZZZ{\JJ}{\cB}{n}{N}}{\rho}\leq
\max(\norm{\ZZZ{\JJ}{\cB}{n}{N}-\ZZZ{\JJ}{\cB}{n-1}{N}}{\rho},\norm{\ZZZ{\JJ}{\cB}{n-1}{N}}{\rho})
\leq \norm{\ZZZ{\JJ}{\cB}{n-1}{N}}{\rho}$$
et la suite
$\norm{\ZZZ{\JJ}{\cB}{n}{N}}{\rho}$ est bornée par rapport à $n$.

Le déterminant de $\ZZZ{\JJ}{\cB}{n}{N}$ est égal à
$ \prod_{s=1}^{d}\left (\prod_{m=N}^n \txin{m}{\inclus{\II{t_{\HT,s}}{t_{\HT,d}}}{\JJ}}\right )$.
En passant à la limite, on en déduit que le déterminant de $\Zinftyu{\JJ}{\cB}{N}$
est égal à
$ \prod_{s=1}^{d}\mxiinf{N}{\inclus{\II{t_{\HT,s}}{t_{\HT,d}}}{\JJ}}$.
\end{proof}
\begin{thm}\label{zndivbis}
\label{zndiv3}Le sous-$\Lambda\ot$-module
$\Module{\JJ}{\cB}{N}{\Lambda\ot\otimes\DD}$\index{$\Module{}{\cB}{N}{\Lambda\ot\otimes\DD}$}
  de $\cH\oEE \otimes\oEE \DD$ engendré par l'image de $\DD$ par
  $\Zinftyu{\JJ}{\cB}{N}$ est contenu dans $\LLL{}{N}{\DD}$.
Sa pente $\slope{\Module{\JJ}{\cB}{N}{\DD}}$
vérifie
$$t_{\Ne,1}-t_{\HT,d}\leq
\slope{\Module{\JJ}{\cB}{N}{\DD}} \leq
\ord_p(\norm{\varphi^{-1}}{\cB})-t_{\HT,1}+ \tborne{\lgg{\JJ}}.$$
\end{thm}
\begin{rem}
Le module $\Module{\JJ}{\cB}{N}{\DD}$ se comporte bien par twist
et les encadrements trouvés sont bien invariants par twist.
\end{rem}
\begin{proof}
Pour montrer que l'image de $\Zinftyu{\JJ}{\cB}{N}$ est contenue
dans $\LLL{}{N}{\DD}$,
calculons $\phiun^{\nm} \Zinftyu{\JJ}{\cB}{N}(u^{j}\zeta - 1)$
pour $n \geq N$, $\zeta$ une racine de l'unité
d'ordre $p^n$ et $j$ appartenant à $]t_{\HT,1},t_{\HT,d}]$:
\begin{equation*}
\begin{split}
\phiun^{\nm} \Zinftyu{\JJ}{\cB}{N}(u^{j}\zeta-1)&=
\phiun^{\nm} \ZZZ{\JJ}{\cB}{n}{N}(u^{j}\zeta-1)\\
&=
 \TH{\Delta_\cB}{n}{\JJ}(u^{j}\zeta-1)
 \phiun^{\nm}\ZZZ{\JJ}{\cB}{n-1}{N}(u^{j}\zeta-1)
\end{split}
\end{equation*}
Pour $i\in \{t_{\HT,1},\cdots,t_{\HT,d}\}$ et $t_{\HT,1}< j\leq t_{\HT,d}$,
la restriction de $\TH{\Delta_\cB}{n}{\JJ}(u^{j}\zeta-1)$ à $\Delta_\cB^{i}$ est
$\txin{n}{\inclus{\III{{i}}}{\JJ}}(u^{j}\zeta-1)\id_{\Delta_\cB^i}$.
Pour $j>i$, $\III{i}$ contient $j$,
$\txin{n}{\inclus{\III{i}}{\JJ}}(u^{j}\zeta-1)$ est nul.
Donc $\TH{\Delta_\cB}{n}{\JJ}(u^{j}\zeta-1){| \Delta_\cB^i}$ est nul.
L'image de $\TH{\Delta_\cB}{n}{\JJ}(u^{j}\zeta-1)$ est donc contenue dans
$K_n\otimes \left(\oplus_{i\geq j} \Delta_\cB^i\right)=K_n\otimes \Fil^j \DD$.
Il en est de même de l'image de
$\phiun^{\nm} \Zinftyu{\JJ}{\cB}{N}(u^{j}\zeta-1)$.

Montrons les assertions sur la pente.
Prenons $\rho=\rho_0$
et $\rho_n=\rho_0^{1/p^n}$.
Dans ce qui suit, les constantes introduites ne dépendent pas de $n$.
On a
$$\phiun^{\nm}\Zinftyu{\JJ}{\cB}{N} = \phiun^{\nm}\ZZZ{\JJ}{\cB}{n}{N} +
\sum_{m=n}^\infty \phiun^{\nm}(\ZZZ{\JJ}{\cB}{m+1}{N}-\ZZZ{\JJ}{\cB}{m}{N}).$$
Posons $B_{n,\cB}=\TH{\Delta_\cB}{n}{\JJ} \circ\phiun^{-1}$.
On a pour $n\geq N$
 \begin{equation*}
\begin{split}
\phiun^{\nm}\ZZZ{\JJ}{\cB}{n}{N}&=B_{n,\cB} \cdots B_{N,\cB}
\\
\phiun^{\nm}\left(\ZZZ{\JJ}{\cB}{n}{N}-\ZZZ{\JJ}{\cB}{n-1}{N}\right)&=
  \left(\TH{\Delta_\cB}{n}{\JJ} - \id\right)
\phiun^{-1}B_{n-1,\cB} \cdots B_{N,\cB}
.\end{split}
\end{equation*}
D'après la proposition \ref{prop:majoration},
il existe un entier $\gamma(J)$ tel que
$\norm{\txin{m}{\inclus{\III{i}}{\JJ}}}{\rho_n} \leq 1$
pour $m \geq \gamma(J) + n_0(\rho)$.
Pour $n > n_0(\rho)$, on a
\begin{equation*}
\begin{split}
\norm{B_{n-1,\cB} \cdots B_{n_0(\rho),\cB}}{\cB,\rho}
&\leq \left(\prod_{m=n_0(\rho)}^{n_0(\rho)+\gamma(\JJ)-1}
\max_i\norm{\txin{m}{\inclus{\III{i}}{\JJ}}}{\rho}\right)
\norm{\varphi^{-1}}{\cB}^{n-n_0(\rho)}\\
&\leq p^{(\tborne{\lgg{\JJ}}+t_{\HT,d}-t_{\HT,1})\gamma(\JJ)}
\norm{\varphi^{-1}}{\cB}^{n-n_0(\rho)}
\end{split}
\end{equation*}
et
$$\norm{B_{n-1,\cB} \cdots B_{1,\cB}}{\cB,\rho}
\leq p^{(\tborne{\lgg{\JJ}}+t_{\HT,d}-t_{\HT,1})(\gamma(\JJ)+n_0(\rho)}
\norm{\varphi^{-1}}{{\cB}}^{n}
.$$
On a donc pour $m \geq n$,
\begin{equation*}
\begin{split}
\norm{\phiun^{\nm}(\ZZZ{\JJ}{\cB}{m+1}{N}-\ZZZ{\JJ}{\cB}{m}{N}) }{\cB,\rho_n}
&\leq C_4 \norm{\varphi^{m-n}}{\cB}
\norm{\TH{\Delta_\cB}{m}{\JJ}-\id}{\cB,\rho_n}\norm{\varphi^{-1}}{\cB}^{m-n}\norm{B_{n,\cB}
\cdots B_{1,\cB}}{\cB,\rho_n}\\
&\leq C_5 p^{-(m-n)(\lgg{\JJ} - \ord_p(\norm{\varphi}{}\norm{\varphi^{-1}}{}))}
\norm{B_{n,\cB} \cdots B_{1,\cB}}{\cB,\rho_n}\\
&\leq C_5\norm{B_{n,\cB} \cdots B_{1,\cB}}{\cB,\rho_n}= C_5
\norm{\phiun^{\nm}\ZZZ{\JJ}{\cB}{n}{N}}{\cB,\rho_n}.\end{split}
\end{equation*}
Il reste donc à majorer la norme de
$
\phiun^{\nm}\ZZZ{\JJ}{\cB}{n}{N} = B_{n,\cB} \cdots B_{1,\cB}
$. On déduit de la proposition \ref{const:norme} et de l'écriture
$B_{k,\cB}=
\TH{\Delta_\cB}{k}{\JJ}\circ\phiun^{-1}$
que
$$\norm{B_{k,\cB}}{\cB,1} \leq p^{\tborne{\lgg{\JJ}}+t_{\HT,d}-t_{\HT,1}}
\norm{\varphi^{-1}}{\cB}.$$
Donc,
 \begin{equation*}
\norm{\phiun^{\nm}\Zinftyu{\JJ}{\cB}{N}}{\cB,\rho_n}
\leq C_5 p^{n\left(\tborne{\lgg{\JJ}}+t_{\HT,d}-t_{\HT,1}
+\ord_p(\norm{\varphi^{-1}}{\cB})\right)}.
\end{equation*}
La suite
$\norm{p^{n \left(\tborne{\lgg{\JJ}}+t_{\HT,d}-t_{\HT,1}+\ord_p(\norm{\varphi^{-1}}{\cB})\right)}
 \phiun^{-n}\Zinftyu{\JJ}{\cB}{N}}{\cB,\rho_n}$ est bornée.
Donc le sous-module $\Module{\JJ}{\cB}{N}{\DD}$
de $\cH\otimes \DD$ est de pente inférieure ou égale à
$
\ord_p(\norm{\varphi^{-1}}{\cB})-t_{\HT,1}+\tborne{\lgg{\JJ}}
$.

La pente de $\Zinftyu{\JJ}{\cB}{N}$ est supérieure à
$t_{\Ne,1}-t_{\HT,d}$.
En effet, si $\norm{\varphi^{\np}}{\cB} =p^{-s_n}$,
$s_n$ est le plus grand rationnel tel que $\varphi^{\np} M_{\cB} \subset p^{s_n} M_{\cB}$,
ce qui est équivalent à $p^{-s_n} M_{\cB} \subset \varphi^{\nm} M_{\cB}$.
On a alors
$$\norm{p^{-s_n}\Zinftyu{\JJ}{\cB}{N}}{\cB,\rho_n}
\leq \norm{\phiun^{\nm}\Zinftyu{\JJ}{\cB}{N}}{\cB,\rho_n}.$$
Soit $h$ un entier tel que les
$\norm{p^{hn}\phiun^{\nm}\Zinftyu{\JJ}{\cB}{N}}{\cB,\rho_n}
$ sont bornés. Alors, les $\norm{p^{(h-s)n}\Zinftyu{\JJ}{\cB}{N}}{\cB,\rho_n}$
sont bornés avec $s=\limsup\frac{s_n}{n}=-t_{\Ne,1}$.
Comme $\Zinftyu{\JJ}{\cB}{N}(\DD)$ est contenu dans $\cH\otimes \DD$,
on a $h-s\geq 0$. Par définition de la pente,
cela implique que la pente de $\Zinftyu{\JJ}{\cB}{N}$ est supérieure à
$t_{\Ne,1}-t_{\HT,d}$. On passe facilement à l'assertion sur la pente de
$\Module{\JJ}{\cB}{N}{\DD}$.
\end{proof}
Sous une hypothèse supplémentaire sur $\cB$, il est possible de préciser
la pente.
\begin{thm}\label{thm:smith}
Supposons que la base $\cB$ est adaptée au $\varphi$-module filtré $\DD$
et soient $\sm{1},\cdots, \sm{d}$ les invariants de Smith du
$\oZE$-module engendré par $\cB$.
Alors, la pente $\slope{\Module{\JJ}{\cB}{N}{\DD}}$
du $\Lambda_{\oZE}$-module $\Module{\JJ}{\cB}{N}{\DD}$
vérifie
$$
{\min}_{1\leq i \leq d} (\sm{i} -t_{\HT,i})
\leq \slope{\Module{\JJ}{\cB}{N}{\DD}} \leq
\max_{1\leq i \leq d} (\sm{i}-t_{\HT,i})+ \tborne{\lgg{\JJ}}.$$
\end{thm}
\begin{proof}On reprend les notations de la démonstration des théorèmes précédents.
Dans la base $\cB$, la matrice de $\varphi^{-1}$ est de la forme
$diag(p^{-\sm{1}},\cdots, p^{-\sm{d}}) U$ avec $U\in\Gl_d(\oZE)$.
La matrice de $B_{n,\cB}$ est
$$
diag(p^{-\sm{1}}\txin{n}{\inclus{\III{t_{\HT,1}}}{\JJ}},\cdots,
p^{-\sm{d}}\txin{n}{\inclus{\III{t_{\HT,d}}}{\JJ}}) U$$
et
$$
p^{\min_{1\leq i\leq d} (\sm{i}-t_{\HT,i}+t_{\HT,d})}\leq \norm{B_{n,\cB}}{\cB,1}
\leq p^{\max_{1\leq i\leq d} (\sm{i}-t_{\HT,i}+t_{\HT,d}+\tborne{\lgg{\JJ}})}
$$
(si $D$ est une matrice diagonale, on a $\norm{DU}{1}=\norm{D}{1}$).
En reprenant le calcul de la pente dans la démonstration précédente, on en déduit
le théorème.
\end{proof}
\begin{cor}\label{cor:divisible}
Si $\DD$ est un $\varphi$-module filtré faiblement admissible et $\cB$
une base adaptée au $\varphi$-module filtré engendrant un réseau
fortement divisible,
on a
$$0\leq \slope{\Module{\JJ}{\cB}{N}{\DD}} \leq \tborne{\lgg{\JJ}}.$$
En particulier, si $\tborne{\lgg{\JJ}}=0$, $\slope{\Module{\JJ}{\cB}{N}{\DD}}$
est nul.
\end{cor}
\begin{proof} La base $\cB$ engendrant un réseau fortement divisible,
le polygone de Smith associé à ce réseau est égal au polygone de Hodge-Tate de
la filtration.
\end{proof}
On construit de même des sous-$\Lambda\ot$-modules
de $\LLL{\JJ}{-N}{\DD}$ pour $N\geq 1$.
\begin{enumerate}
\item
Soit $\Zinftyu{\JJ}{\cB}{-N}$ l'endomorphisme construit
comme $\Zinftyu{\JJ}{\cB}{N}$ à l'aide de la suite définie par
$$\ZZZ{\JJ}{\cB}{n}{-N}=\phiun^{\np} \TH{\Delta_\cB}{n}{\JJ}
\phiun^{\nm} \ZZZ{\JJ}{-N}{n-1}{\cB}$$
pour $n \geq N$ et
$\ZZZ{\JJ}{-N}{N-1}{\cB}= \bomega{N-1}{\JJ}\id$.
Le sous-$\Lambda\ot$-module $\Module{\JJ}{\cB}{-N}{\DD}$
de $\cH\oEE\otimes \DD$ associé est contenu dans $\LLL{\JJ}{-N}{\DD}$.
\item
Supposons que $\varphi$ n'admet de valeur propre de la forme $p^j$
et prenons la suite de premier terme
$$\ZZZ{\JJ}{\cB}{0}{eul}=\phiun_{eul}\TH{\Delta_\cB}{0}{\JJ}\phiun_{eul}^{-1}$$ où
 $\phiun_{eul}$ est l'endomorphisme de $\oZE[x]\otimes \DD$
de degré strictement inférieur à $\lgg{\JJ}$ tel que
$$\phiun_{eul}(u^j-1) = (1-p^{-j}\phiun)(1-p^{-1+j}\phiun^{-1})^{-1}
$$ pour $j \in \JJ$.
Le sous-$\Lambda\ot$-module obtenu $\Module{\JJ}{\cB}{eul}{\DD}$
est contenu dans $\Module{\JJ}{\cB}{1}{\DD}$.
Il tient compte des facteurs d'Euler "classiques".
\end{enumerate}
\begin{prop}Soit $N$ un entier $\geq 0$.
\begin{enumerate}
\item
L'application $\theta_N :
\Module{\JJ}{\cB}{N+1}{\DD}/\Module{\JJ}{\cB}{N}{\DD} \to
  \prod_{j\in \III{t_{\HT,1}}} K_N \otimes (\DD/\Fil^{j}\DD)$
est surjective.
\item
 L'application $\theta^N :
 \Module{\JJ}{\cB}{-N}{\DD}/\Module{\JJ}{\cB}{-N+1}{\DD} \to
  \prod_{j\in \III{t_{\HT,1}}} K_N \otimes \Fil^{j}\DD$ est surjective
si dans le cas où $\newtH\neq t_{\HT,d}$,
$\txin{N}{\inclus{\III{t}}{\JJ}}(u^{t_{\HT,1}+1} \zeta_N -1)$ est non nul.
\end{enumerate}
\end{prop}
\begin{proof}
Soit $j\in \III{t_{\HT,1}}$.
On a $\Zinftyu{\JJ}{\cB}{N+1}(u^j \zeta_N-1)
=\ZZZ{\JJ}{\cB}{N}{N+1}(u^j \zeta_N -1)=\id$.
D'où la surjectivité de $\theta_N$.
La surjectivité de $\theta^N$ vient de ce que
\begin{equation*}
\begin{split}
\phiun&^{-(N+1)}\Zinftyu{\JJ}{\cB}{-N}(u^j \zeta_N -1)
=\phiun^{-(N+1)}\ZZZ{\JJ}{\cB}{N}{-N}(u^j \zeta_N -1)
\\
&=
diag_{v}([\txin{N}{\inclus{\III{t_{\HT,v}}}{\JJ}}(u^j \zeta_N-1)]_{v})
\phiun^{-(N+1)} \prod_{j'\in \III{t_{\HT,1}}}(u^{(j-j')p^{N-1}}
  \zeta_N^{p^{N-1}} -1).
\end{split}
\end{equation*}
\end{proof}
\begin{rem}
Le déterminant de $\Zinftyu{\JJ}{\cB}{N}$ est un multiple de
$\prod_{j=1}^{d}(\mloginf{[j]}{N})^{\dim \DD/\Fil^{j}\DD}$.
C'est une fonction d'ordre compris entre
$t_{\HT}(\DD)$ et $t_{\HT}(\DD) + \tborne{\lgg{\JJ}} \times \dim \DD$.
Vu comme un élément du $\varphi$-module filtré $\cH\otimes \det\DD$,
il est de pente compris entre 0 et $\tborne{\lgg{\JJ}} \times \dim \DD$.
\end{rem}

\subsection{Construction associée à un raffinement}
\label{raffinement}
\def\No{N}
Nous allons revoir la construction de manière à mettre en évidence sa dépendance
par rapport au $\varphi$-module filtré.
Pour simplifier, supposons dans ce paragraphe que les poids de Hodge-Tate de
$\DD$ sont distincts et que $\varphi$ est diagonalisable, même si ce n'est pas essentiel.
En particulier, $\Fil^{t_{\HT,i}}\DD$ est de dimension $d+1-i$.

On se donne un raffinement $\cF_\bullet$ du $\varphi$-module filtré $\DD$.
Soit
 $\cB=(v_1,\cdots, v_d)$ une base de $\DD$ adaptée à ce raffinement :
$(v_1,\cdots, v_d)$ est adaptée à la filtration $\Fil^\bullet$
et $\cF_i=\langle v_1,\cdots, v_i\rangle$ est stable par $\varphi$.
Soit $\cB_0=(e_1,\cdots, e_d)$ une base de $\DD$ formée de vecteurs propres
telle que $(e_1,\cdots, e_i)$ soit une base de $\cF_i$.
La matrice $P$ de $\cB$ dans la base $\cB_0$
est triangulaire supérieure et on peut la supposer unipotente:
$P =((\lambda_{i,j}))$ avec $\lambda_{i,j}=0 $ si
$i>j$ et $\lambda_{i,i}=1$.
\begin{thm} \label{znraff} Soit $\cF_\bullet$ un raffinement
du $\varphi$-module filtré $\DD$ et $\cB$ et $\cB_0$
des bases comme ci-dessus.
Soit $(\alpha_1, \cdots, \alpha_d)$ la suite
des valeurs propres associée à $\cF_\bullet$:
$\varphi e_i=\alpha_i e_i$.
Soit $J$ un intervalle fini de $\ZZ$ contenant $]t_{\HT,1},t_{\HT,d}]$
et vérifiant $\lgg{\JJ} > \max_{ j \geq i} \ord_p(\alpha_j/\alpha_i)$.
L'endomorphisme $\Zinftyu{\JJ}{\cB}{\No}$ est un polynôme de degré $\leq d-1$
en les $\lambda_{i,j}$ pour $i<j$.
La pente de $\Module{\JJ}{\cB}{\No}{\DD}$ est inférieure ou égale à
$\tborne{\lgg{\JJ}} + \max_{i\leq k}(\ord_p(\alpha_k)-t_{\HT,i})$.
\end{thm}
\begin{proof}
Travaillons pour $n\geq \No$ avec $\ZZZ{\JJ}{\cB}{n}{N}$ et notons $Z_n$
sa matrice dans la base $\cB_0$.
La matrice $T_{n}$ de $\TH{\cB}{n}{\JJ}$ dans la base $\cB_0$
est égale à $P diag([\txin{n}{\inclus{\III{t_{\HT,i}}}{\JJ}}]_i)P^{-1}$ et
a comme coefficients
$$
t_{i,j}^{(n)} = \begin{cases}
0 &\text{ pour $i>j$}\\
\txin{n}{\inclus{\III{t_{\HT,i}}}{\JJ}} &\text{ pour $i=j$}\\
\sum_{k=i}^{j}\lambda_{i,k} \lambda_{k,j}'
( \txin{n}{\inclus{\III{t_{\HT,k}}}{\JJ}} -1) &\text{ pour $i<j$}
\end{cases}
$$
où les $\lambda_{i,j}'$ sont les coefficients de la matrice $P^{-1}$.
La matrice $\VV{\cB_0}$ de $\varphi$ dans la base $\cB_0$ est
la matrice diagonale de diagonale $(\alpha_1,\cdots, \alpha_d)$.
La matrice $Z_n$ est donc
une matrice triangulaire supérieure dont les coefficients diagonaux sont
$$(\prod_{m=\No}^{n} \txin{m}{\inclus{\III{t_{\HT,i}}}{\JJ}})_{i=1, \cdots, d}$$
et dont les coefficients $\coeffZ{i,j}{n}$ pour $i<j$ vérifient les relations de récurrence
$$
\coeffZ{i,j}{n} = \sum_{i\leq l \leq j }
 (\frac{\alpha_i}{\alpha_l})^{\np}t_{i,l}^{(n)} \coeffZ{l,j}{n-1}
= \txin{n}{\inclus{\III{t_{\HT,i}}}{\JJ}}\coeffZ{i,j}{n-1} +
\sum_{i<l \leq j } (\frac{\alpha_i}{\alpha_l})^{\np}t_{i,l}^{(n)}
\coeffZ{l,j}{n-1}
.$$
Donnons d'abord le lemme suivant.
 \begin{lem}\label{lempq}
Soit $A$ un anneau intègre.
Soient $P$ et $Q$ deux matrices triangulaires supérieures à coefficients dans
l'anneau de polynômes $A[x_1,\cdots, x_n]$ vérifiant la propriété suivante :

(*) pour tout $i\leq j$, le coefficient $(i,j)$ est un polynôme de degré
 total inférieur ou égal à $j-i$.
\begin{enumerate}
 \item Le produit $P Q$ vérifie aussi la condition (*).
 \item Si de plus $P$ est unipotente et de degré $\leq 1$, $P^{-1}$ vérifie
 la condition (*).
\end{enumerate}
\end{lem}
\begin{proof}
Pour la première assertion, il suffit d'écrire la formule du produit.
Pour la seconde assertion, on écrit $P=\id + N$ avec $N$ matrice triangulaire
supérieure nilpotente. On a donc
$P^{-1}= \id - N + N^2- \cdots + (-1)^{d-1} N^{d-1}$.
Le coefficient $(i,j)$ de $P^{-1}$ est le coefficient $(i,j)$ de la matrice
$\id - N + N^2 - \cdots (-1)^{d-1} N^{j-i}$ qui est de degré $\leq j-i$.
\end{proof}
Continuons la démonstration du théorème \ref{znraff}.
On applique le lemme \ref{lempq} à $P^{-1}$, au produit de
$P$ par $diag([\txin{n}{\inclus{\III{t_{\HT,k}}}{\JJ}}]_k) P^{-1}$
et au produit de telles matrices pour $m\leq n$.
On en déduit que pour $i \leq l$, les $t_{i,l}^{(n)} $ sont des polynômes
en les $\lambda_{k,k'}$ de degré total inférieur ou égal à $l-i$ et à coefficients
dans le $\Zp$-module de $\cH$ engendré par les
 $ \txin{n}{\inclus{\III{t_{\HT,h}}}{\JJ}} -1 $ pour $i \leq h \leq l$
 et que le coefficient $\coeffZ{i,j}{n}$ de $Z_n$
est un polynôme en les $\lambda_{k,l}$ de degré total inférieur ou égal à
$j-i\leq d-1$.
On a donc pour $i \leq l$ et $n$ assez grand,
$$ \norm{(\frac{\alpha_i}{\alpha_l})^{\np}t_{i,l}^{(n)}}{\rho} \leq C(\rho)
p^{-n \left(\lgg{\JJ} + \ord_p(\alpha_i/\alpha_l)\right)}$$
et pour $i \leq j$,
$$\norm{\coeffZ{i,j}{n} - \coeffZ{i,j}{n-1}}{\rho}  \leq
 C'(\rho) p^{ -n(\lgg{\JJ} - \max_{ i\leq l \leq j} \ord_p(\alpha_l/\alpha_i))}$$
où $C(\rho)$ et $C'(\rho)$ sont des constantes faisant intervenir $\norm{P}{1}^d$ avec
$\norm{P}{1}=\max(\abs{\lambda_{k,k'}}{p}$).
On en déduit la convergence de $Z_n$.
Plus précisément, la suite $(\coeffZ{i,j}{n})_n$ converge si
$ \lgg{\JJ} > \max_{i\leq l \leq j}\ord_p(\alpha_l/\alpha_i)$.
Sa limite est encore un polynôme
en les $\lambda_{k,l}$ de degré total inférieur ou égal à
$j-i$ et la matrice $Z_\infty$ de
$\Zinftyu{\JJ}{\cB}{\No}$ dans la base $\cB_0$ est de degré total
en les $\lambda_{k,l}$ inférieur ou égal à $d-1$.

Le calcul de la pente de $\Zinftyu{\JJ}{\cB}{\No}$ se ramène
au comportement des
$\norm{\coeffZ{i,j}{n}}{1}$ comme dans la démonstration du théorème \ref{zndiv}.
Écrivons
$(\UU{\cB_0}{\np} T_{n})_{i,j}=
\sum_{\deg(s_\kappa)\leq j-i} c_{\kappa,n}^{(i,j)} \lambda_{\kappa}^{s_\kappa}$
avec $\kappa$ multi-indice. On a
$$\norm{c_{\kappa,n}^{(i,j)}}{1} \leq p^{\tborne{\lgg{\JJ}} +t_{\HT,d} -t_{\HT,i}+
\max_{i\leq k\leq l\leq j}\ord_p(\alpha_l)-t_{\HT,k}+t_{\HT,i}}
\leq p^{\tborne{\lgg{\JJ}} +t_{\HT,d} -t_{\HT,i}+m_{i,j}}$$ avec
$
m_{i,j}=\max_{i \leq l\leq j}\ord_p(\alpha_l)$.
On déduit de ce qui précède que
$$
\norm{\alpha_i^{\nm} \coeffZ{i,j}{n}}{1}
\leq \norm{P}{1}^d p^{n(\tborne{\lgg{\JJ}}+t_{\HT,d} -t_{\HT,i}) + m_{i,j}^{(n)}}
$$
où $m_{i,j}^{(n)}$ vérifie
$$ m_{i,j}^{(1)}= m_{i,j},\
  m_{i,j}^{(n)} \leq \max_{i\leq l\leq j} m_{i,l}^{(n-1)} + m_{i,j}
.$$
Montrons par récurrence que
$$m_{i,j}^{(n)} \leq \left(\max_{i\leq k\leq j} m_k\right) n$$
avec $m_k=\ord_p(\alpha_k)$.
La relation est bien vraie pour $n=\No$. Si elle est vraie pour $n-1$, on a
\begin{equation*}
\begin{split}
m_{i,j}^{(n)}
&\leq \max_{i\leq l\leq j}\left(\max_{i\leq k\leq l}m_l\right)(n-1)+
\max_{i\leq l\leq j} m_l\\
&
  \leq \left(\max_{i\leq k\leq j}m_k\right)n.
\end{split}
\end{equation*}
On en déduit que pour $n$ assez grand,
$$
\norm{\alpha_i^{\nm} \coeffZ{i,j}{n}}{}
\leq C\norm{P}{1}^dp^{\left(\tborne{\lgg{\JJ}}+t_{\HT,d}-t_{\HT,i}+
  \max_{i\leq k\leq j}(\ord_p(\alpha_k))\right)n}
$$
et que la pente de $\Zinftyu{\JJ}{\cB}{\No}$ est inférieure ou égale à
$\tborne{\lgg{\JJ}} + \max_{i\leq k}(\ord_p(\alpha_k)-t_{\HT,i})$.
\end{proof}

\subsection{Exemple de la dimension 2}
Prenons pour $\DD$ un
$\varphi$-module filtré de dimension 2 de poids de Hodge $t_{\HT,1}=-r<0$, $t_{\HT,2}=0$
dont le polynôme caractéristique de $\varphi$
est $x^2 - p^{-r} a_p x + p^{-r} \U$ avec $a_p \in \oZE$
et $\U$ une unité.
On note $\omega$ une base de $\Fil^0\DD$.
Supposons de plus que le $\varphi$-module filtré $\DD$ est indécomposable,
ce qui implique que $\omega$ n'est pas vecteur propre de $\varphi$.

\subsubsection{Base adaptée au $\varphi$-module filtré}
Une base adaptée au $\varphi$-module filtré $\DD$
est $\cB= (\varphi(\omega),\omega)$. En effet,
la matrice de $\varphi$ dans la base $\cB$ est
\begin{equation*}
\begin{split}
\varphi_{\cB}&=
\smallmat{ \frac{a_p}{p^{r}}&1\\-\frac{\U}{p^{r}}&0}
=\smallmat{a_p&1\\-\U&0}\smallmat{ p^{-r}&0\\0&1}
\end{split}
\end{equation*}
avec $\smallmat{a_p&1\\-\U&0}\in \Gl_2(\oZE)$ et
est de norme $p^{r}$.
Son inverse $\varphi_{\cB}^{-1} =\U^{-1}\smallmat{0&-p^{r}\\\U&a_p}$
est de norme 1.
Le réseau $\res{\cB}= \oZE \varphi \omega \oplus \oZE\omega$ est
un réseau fortement divisible.
Soient $t_{\Ne,1} \leq t_{\Ne,2}$ les pentes de Newton de $\varphi$.
Ce sont aussi les valuations $p$-adiques de ses valeurs propres.
On a
$$\sm{1}=t_{\HT,1}=-r, \quad \sm{2}=t_{\HT,2}=0,\quad
-r\leq t_{\Ne,1}\leq t_{\Ne,2}\leq 0.$$
Si $p$ divise $a_p$, on a $-r<t_{\Ne,1} \leq t_{\Ne,2}<0$
et $\newtHd=\rabiot$ avec $\rabiot=1$ si $p$ ne divise pas $a_p$ et
$0$ sinon. Prenons $J=]-r,\rabiot]$.
On a
\begin{equation*}
\ZZZ{J}{\cB}{n}{N} = \varphi_{\cB}^{\np}
\smallmat{\txin{n}{\inclus{\II{-r}{0}}{J}}&0\\0&1}
\varphi_{\cB}^{\nm}\ZZZ{J}{\cB}{n-1}{N}
\end{equation*}
D'après le théorème \ref{zndiv}, la suite de matrices $\ZZZ{J}{\cB}{n}{N}$ converge.
La pente de $\Zinftyu{J}{\cB}{N}$
est comprise entre $0$ et $\tborne{\lgg{J}}$. Si $\tborne{\lgg{J}}=0$
(par exemple $p > \lgg{J}$),
$\mxiinf{N}{\inclus{\II{-r}{0}}{J}}$
est égal à une unité près à $\mloginf{\II{-r}{0}}{N}$.
\begin{rem}Pour $a_p=0$, on a
$$
B_{m+1,\cB} B_{m,\cB} =
\smallmat{
-p^{r}\U \txin{m+1}{\II{-r}{0}}&0\\
0&
-p^{r}\U \txin{m}{\II{-r}{0}}}=
- p^r\U\smallmat{\txin{m+1}{\II{-r}{0}}&0\\0&\txin{m}{\II{-r}{0}}}
=\varphi^{-2}\smallmat{\txin{m+1}{\II{-r}{0}}&0\\0&\txin{m}{\II{-r}{0}}}
$$
Par passage à la limite, on retrouve sur la diagonale de $\Zinftyu{}{\cB}{N}$ l'analogue
$\prod_{m \text{ pair}}\txin{m}{\II{-r}{0}}$ et
$\prod_{m \text{ impair}}\txin{m}{\II{-r}{0}}$ des éléments
$\log^+$ et $\log^-$ de \cite{pollack}.
\end{rem}
Si la matrice de $\Zinftyu{}{\cB}{1}$ dans la base $\cB=(\varphi(\omega),\omega)$ est
$\smallmat{ Z_{11}&Z_{12}\\Z_{21}&Z_{22}}$,
un élément de $\Module{}{\cB}{1}{}$ est de la forme
$$
\Psi_\cB(f,g)=f(Z_{11} \varphi\omega + Z_{21} \omega) + g(Z_{12} \varphi\omega + Z_{22}\omega)
= (f Z_{11} + g Z_{12}) \varphi\omega + (f Z_{21} + g Z_{22})\omega
$$
avec $f$ et $g$ appartenant à $\Lambda\ot$.
\begin{prop} Soit $(f, g) \in \cH\oEE^2$. Alors, $\Psi_\cB(f,g)$
appartient à $\LLL{}{Eul}{\DD}$ si et seulement si
pour tout $-r \leq j < 0$,
$$
(1-p^{-1}\U) f(u^{j}-1) +(p^{r+j-1}+ p^{-j} -p^{-1}a_p)g(u^{j}-1)= 0
.$$
\end{prop}
\begin{proof}
Pour $-r\leq j < 0$, la valeur en $u^{j}-1$ de $\Zinftyu{}{\cB}{1}$ est
l'identité. On a d'autre part
\begin{equation*}
\begin{split}
(1-p^{-j}&\varphi_{\cB})^{-1}(1-p^{j-1} \varphi_{\cB}^{-1})=\\
&\frac{p^{j+r}}{p^{j+r} +  p^{-j}\U - a_p}
\smallmat{
1-p^{-1}\U&
p^{r+j-1} + p^{-j} -p^{-1}a_p\\
- p^{-(j+r)}\U - p^{j-1}\U + p^{-(r+1)}a_p &
 p^{-1}\U+1+(p^{-(j+r)} + p^{j-1})a_p +p^{-(r+1)}a_p^2
}.
\end{split}
\end{equation*}
Le second vecteur de la base $\cB$ étant une base de
$\Fil^{-r+1}\DD= \cdots = \Fil^{0}\DD$,
la condition s'en déduit.
\end{proof}

\subsubsection{Base adaptée à un raffinement}
On suppose pour simplifier que les valeurs propres sont dans $\EEE$.
Soit $\cB_0=(e_1,e_2)$ avec
$\varphi e_1 = \alpha_1 e_1$,
$\varphi e_2 = \alpha_2 e_2$ où $\alpha_i$ est de valuation $p$-adique $s_i$
avec $-r\leq s_1, s_2 \leq 0$, $s_1 + s_2 =-r$
(on a $\{s_1,s_2\}=\{t_{\Ne,1},t_{\Ne,2}\}$).
Écrivons $\omega \in\Fil^0 \DD$ dans la base $\cB_0$ normalisée
de manière à ce que $\omega = -\lambda e_1 + e_2$ pour $\lambda \in \oZE$
(éventuellement $\lambda \neq 0$ selon le choix de $e_1$ et $e_2$).
Comme $\DD$ est indécomposable, $\omega$ n'est pas vecteur propre de $\varphi$ et
$\cB=(e_1,\omega)$ est une base de $\DD$ à la filtration de $\cD$
et adaptée au raffinement $\EEE e_1 \subset \DD$.
Le changement de base entre les bases $\cB_0$ et $\cB$ est donné par la matrice
$P=\smallmat{1&-\lambda\\
0&1
}$. La matrice $\VV{\cB}$ de $\varphi$ dans la base $\cB$ est
$$\VV{\cB}=P^{-1} \smallmat{\alpha_1 &0\\
0&\alpha_2
} P=\smallmat{\alpha_1&(\alpha_2-\alpha_1)\lambda\\
0&\alpha_2
},\quad \UU{\cB}{-1}=
\smallmat{\alpha_1^{-1}&(\alpha_2^{-1}-\alpha_1^{-1})\lambda\\
0&\alpha_2^{-1}}.
$$
On a
$$
\smallmat{\alpha_1&(\alpha_2-\alpha_1)\lambda\\0&\alpha_2}=
\smallmat{\frac{\alpha_1}{p^{s_1}}
&\frac{\alpha_2}{p^{s_2}}(1-\frac{\alpha_1}{\alpha_2})\lambda\\0&\frac{\alpha_2}{p^{s_2}}}
\smallmat{p^{s_1}&0\\0&p^{s_2}}
$$
où $\frac{\alpha_1}{p^{s_1}}$ et $\frac{\alpha_2}{p^{s_2}}$
sont des unités. La base $\cB$ est adaptée au $\varphi$-module filtré
si
$$\ord_p(\lambda) \geq -\ord_p(1-\frac{\alpha_1}{\alpha_2}),$$
ce qui est possible à réaliser quitte à remplacer $e_1$ par un multiple;
elle engendre un $\oZE$-module fortement divisible si de plus
$\{s_1,s_2\}=\{0,-r\}$, c'est-à-dire uniquement si $\cD$ est ordinaire.

Supposons que $\ord_p(\lambda)\geq -\ord_p(1-\frac{\alpha_1}{\alpha_2})$
et appliquons le théorème \ref{thm:smith}.
Lorsque $\DD$ est ordinaire ($s_1=-r$, $s_2=0$),
 la pente de $\Zinftyu{}{\cB}{N}$ est comprise entre
 $0$ et $\tborne{r+1}$
(comme $t_{\HT,d}=0$, c'est aussi son $\varphi$-ordre).
Lorsque $-r < s_1,s_2 <0$ (cas non ordinaire),
on peut prendre $J=]-r,0]$;
la pente de $\Zinftyu{}{\cB}{N}$ est comprise entre $\min(s_1+r,s_2)$
et $\max(s_1+r,s_2)+\tborne{r}$, c'est-à-dire entre $s_2$ et $-s_2+\tborne{r}$.
On voit ainsi l'influence du choix de la base $\cB$ (ici, elle n'engendre pas
un réseau fortement divisible).

Explicitons la construction de la matrice $Z_n$ de
$\ZZZ{\JJ}{\cB}{n}{N}$
dans la base $\cB_0$ avec $J=]-r,\rabiot]$.
On a
\begin{equation*}
Z_n=
\smallmat{
\txin{n}{\inclus{\II{-r}{0}}{\II{-r}{\rabiot}}}&\lambda (\frac{\alpha_1}{\alpha_2})^{\np}
(\txin{n}{\inclus{\II{-r}{0}}{\II{-r}{\rabiot}}}-1)
\\0&1
}
Z_{n-1}
\end{equation*}
avec $Z_N=\id$.
Dans le cas supersingulier, c'est simplement
\begin{equation*}
Z_n=
\smallmat{
\txin{n}{\II{-r}{0}}&\lambda (\frac{\alpha_1}{\alpha_2})^{\np}
(\txin{n}{\II{-r}{0}}-1)
\\0&1
}
Z_{n-1}.
\end{equation*}
Donc, pour $n \geq N$
\begin{equation*}
Z_n
=\smallmat{\prod_{m=N}^{n}\txin{m}{\inclus{\II{-r}{0}}{\II{-r}{\rabiot}}}&\lambda F_{n,\alpha_1/\alpha_2}\\
0&1}
\end{equation*}
où les $F_{n,\nu}$ vérifient la relation de récurrence
$$
F_{n,\nu}= \txin{n}{\inclus{\II{-r}{0}}{\II{-r}{\rabiot}}} F_{n-1,\nu} + \nu^{\np}(\txin{n}{\inclus{\II{-r}{0}}{\II{-r}{\rabiot}}}-1)
$$
avec
$ F_{N,\nu} = 0$
et dans le cas supersingulier
$$
F_{n,\nu}= \txin{n}{\II{-r}{0}} F_{n-1,\nu} + \nu^{\np}(\txin{n}{\II{-r}{0}}-1)
$$
D'où
\begin{equation*}
F_{n,\nu}= \sum_{t=N}^{n} \nu^{t+1}(\txin{t}{\inclus{\II{-r}{0}}{\II{-r}{\rabiot}}}-1)
\prod_{t < i \leq n} \txin{i}{\inclus{\II{-r}{\rabiot}}{\II{0}{\rabiot}}} .
\end{equation*}
On a $F_{n,\nu}(u^{j} \zeta -1) =
F_{m,\nu}(u^{j} \zeta -1)=- \nu^{m+1}$ pour $ j \in \II{-r}{0}$
pour toute racine de l'unité $\zeta$ d'ordre $p^m$ avec $N\leq m\leq n$.
Si $r+\rabiot$ est de plus strictement supérieur à $-\ord_p(\nu)$,
la suite des $F_{n,\nu}$ converge vers un élément
$F_{\infty,\nu}^{\inclus{\II{-r}{0}}{\II{-r}{\rabiot}}}$ de $\cH$
d'ordre inférieur ou égal à $r+\rabiot+\tborne{r+\rabiot}$ (on peut le démontrer directement).
On a alors pour $n \geq 0$, $j\in \II{-r}{0}$ et $\zeta$
racine de l'unité d'ordre $p^n$
\begin{equation*}
F_{\infty,\nu}^{\inclus{\II{-r}{0}}{\II{-r}{\rabiot}}}(u^{j} \zeta -1) = - \nu^{\np}.
\end{equation*}
La matrice $Z_\infty$ de $\Zinftyu{}{\cB}{N}$ dans la base $\cB_0$ est
\begin{equation*}
 Z_\infty
=\smallmat{\txin{\infty}{\inclus{\II{-r}{0}}{\II{-r}{\rabiot}}}&\lambda F^{\inclus{\II{-r}{0}}{\II{-r}{\rabiot}}}_{\infty,\alpha_1/\alpha_2}\\
0&1}
\end{equation*}
ce qui se simplifie dans le cas supersingulier en
\begin{equation*}
 Z_\infty
=\smallmat{\txin{\infty}{\II{-r}{0}}&\lambda F^{\II{-r}{0}}_{\infty,\alpha_1/\alpha_2}\\0&1}
\end{equation*}

\subsection{Exemple de la dimension 3}
Explicitons un peu les calculs dans le cas où $\DD$ est de dimension 3 et où la filtration
est donnée par $0\neq \Fil^{t_{H,3}} \DD \subsetneq \Fil^{t_{H,2}} \DD \subsetneq \Fil^{t_{H,1}} \DD$
en reprenant les notations du paragraphe \ref{raffinement}:
soit $\cB_0=(e_1, e_2, e_3)$ une base de vecteurs propres de $\varphi$
adaptée au raffinement, de valeurs propres respectives $\alpha_1, \alpha_2, \alpha_3$
 et soit $P$ la matrice de la base de la filtration
 $\cB=(v_1, v_2, v_3)$ choisie.
$$P=\smallmat{
1&-\lambda_{12}&-\lambda_{1 3}\\
0&1&-\lambda_{2 3}&\\
0&0&1
}
\quad P^{-1}=\smallmat{
1&\lambda_{1 2}&\lambda_{1 3}+\lambda_{1 2}\lambda_{2,3}\\
0&1&\lambda_{2 3}\\
0&0&1
}
\quad
\VV{\cB_0}=\smallmat{
\alpha_1&0&0\\
0&\alpha_2&0\\
0&0&\alpha_3
}
$$
Du calcul de la matrice de $\varphi$ dans la base $\cB$,
on déduit que
$$P\VV{\cB_0}P^{-1}\VV{\cB_0}^{-1}=
\smallmat{
1&(\frac{\alpha_1}{\alpha_2}-1)\lambda_{12}
&(\frac{\alpha_2}{\alpha_3}-\frac{\alpha_1}{\alpha_3})\lambda_{12}\lambda_{23}
+(\frac{\alpha_1}{\alpha_3}-1)
\lambda_{13}\\
0&1&(1-\frac{\alpha_2}{\alpha_3})\lambda_{23}\\
0&0&1
}.$$
On peut facilement écrire les conditions sur la valuation $p$-adique
des $\lambda_{ij}$ pour que cette matrice soit dans $SL_3(\Zp)$
et donc que la base $\cB$ soit adaptée au $\varphi$-module filtré $\cD$:
les coefficients doivent être des entiers $p$-adiques et il est toujours possible
de modifier les $e_i$ par un scalaire pour que cela soit vérifié
et donc d'appliquer le théorème \ref{thm:smith}.

Faisons le calcul explicite de la matrice dans la base $\cB_0$. Pour simplifier,
supposons que $t'_{\HT,3}=t_{H,3}$ et que  $\JJ=]t_{H,1},t_{H,3}]$, on a
$$\ZZZ{\JJ}{\cB}{n}{N}(\cB_0)= T_{n} \ZZZ{\JJ}{\cB}{n-1}{N}(\cB_0)$$
avec
\begin{equation*}
\begin{split}
&T_{n}=\UU{\cB}{\np}\smallmat{
\txin{n}{\IIII{t_{\HT,1}}}&0&0\\
0&\txin{n}{\inclus{\IIII{t_{\HT,2}}}{\JJ}}&0\\
0&0&\txin{n}{\inclus{\IIII{t_{\HT,3}}}{\JJ}}
}
\UU{\cB}{-n} \\
&=
\smallmat{
\txin{n}{\IIII{t_{\HT,1}}}& (\frac{\alpha_1}{\alpha_2})^{\np}\lambda_{1 2}
(\txin{n}{\IIII{t_{\HT,1}}}-\txin{n}{\inclus{\IIII{t_{\HT,2}}}{\JJ}})&
(\frac{\alpha_1}{\alpha_3})^{\np}\left(\lambda_{13}(\txin{n}{\IIII{t_{\HT,1}}}{}-1)
+\lambda_{1 2}\lambda_{2 3}(\txin{n}{\IIII{t_{\HT,1}}}
-\txin{n}{\inclus{\IIII{t_{\HT,2}}}{\JJ}})\right)\\
0&\txin{n}{\inclus{\IIII{t_{\HT,2}}}{\JJ}}&(\frac{\alpha_2}{\alpha_3})^{\np}
\lambda_{2 3}(\txin{n}{\inclus{\IIII{t_{\HT,2}}}{\JJ}}-1)\\
0&0&1
}
\end{split}
\end{equation*}
On a alors une expression de la forme suivante pour $\ZZZ{\JJ}{\cB}{n}{N}$ :
$$\ZZZ{\JJ}{\cB}{n}{N}= \smallmat{
\prod_{m=N}^n\txin{m}{\IIII{t_{\HT,1}}}& \lambda_{1 2} F_{n,12}&
\lambda_{13} F_{n,13}+\lambda_{1 2}\lambda_{2 3} G_{n,1 3}\\
0&\prod_{m=N}^n\txin{m}{\inclus{\IIII{t_{\HT,2}}}{\JJ}}&\lambda_{2 3} F_{n,2 3}\\
0&0&1
} $$
D'où une relation de récurrence facile à écrire.
En passant à la limite, $\Zinftyu{\JJ}{\cB}{N}$ s'écrit
$$\Zinftyu{\JJ}{\cB}{N}= \smallmat{
\mxiinf{N}{\IIII{t_{\HT,1}}}& \lambda_{1 2} F_{\infty,N,12}&
\lambda_{13} F_{\infty,N,13}+\lambda_{1 2}\lambda_{2 3} G_{\infty,N,1 3}\\
0&\mxiinf{N}{\inclus{\IIII{t_{\HT,2}}}{\IIII{t_{\HT,1}}}}&\lambda_{2 3} F_{\infty,N,2 3}\\
0&0&1
}. $$

\section{Conséquences}
\begin{prop}\label{prop:rang}
Le $\cH\oEE$-module $\LLL{\JJ}{N}{\DD}$ est de rang $d = \dim \DD$.
\end{prop}
\begin{proof}
Le $\cH\oEE$-module $\LLL{\JJ}{N}{\DD}$ contient $\Module{\JJ}{\cB}{N}{\DD}$
qui est de rang supérieur ou égal à $d$
car le déterminant de $\Zinftyu{\JJ}{\cB}{N}$ est non nul.
D'autre part, le module $\LLL{\JJ}{N}{\DD}$ est un sous-module fermé de
$\cH \otimes \DD$ pour la topologie de Fréchet. En effet, si $g_s$ est une suite d'éléments
de $\LLL{\JJ}{N}{\DD}$ convergeant vers un élément $g$ de $\cH\oEE \otimes \DD$, on
a pour tout $j\leq t_{\HT,d}$, $(1\otimes \varphi)^{\np} g_s(u^{j} \zeta -1) \in K_n \otimes \Fil^{j} \DD$
pour tout entier $s$ et pour tout entier $n \geq N$ et $\zeta$ d'ordre $p^n$.
La convergence pour la topologie de Fréchet implique la convergence point par point.
Donc la limite $(1\otimes \varphi)^{\np} g(u^{j} \zeta -1)$
de $(1\otimes \varphi)^{\np} g_s(u^{j} \zeta -1)$ lorsque $s \to \infty$
appartient à $K_n \otimes \Fil^{j} \DD$ qui est un sous-espace fermé de $K_n\otimes \DD$
(dimension finie).
On peut donc appliquer la proposition suivante pour en déduire que $\LLL{\JJ}{N}{\DD}$ est de
rang inférieur ou égal à $d$.
\end{proof}
\begin{prop}[\cite{berger1}]
Soit $M$ un module libre de rang fini $d$ sur $\cH\oEE$ et soit $M'$
un sous-module de $M$ fermé pour la topologie de Fréchet de $M$.
Alors, $M'$ est libre de rang $e\leq d$.
\end{prop}
\begin{prop}\label{detdiv}
Le détermimant de $\LLL{\JJ}{N}{\DD} \to \cH\oEE \otimes\oEE \DD$ (autrement dit
l'indice de $\LLL{\JJ}{N}{\DD}$ dans $\cH\oEE \otimes \DD$)
est divisible par $\prod_{j\leq t_{\HT,d}} (\mloginf{[j]}{N})^{\dim \DD/\Fil^{j}\DD}$.
\end{prop}

On déduit la proposition du lemme suivant :
\begin{lem}
Soit $W$ un $K$-espace vectoriel de dimension $d$ et soient $g_1$, \dots, $g_d$
des éléments de $\cH\oEE\otimes W$. On suppose qu'il existe pour tout entier $n\geq 1$ une
filtration décroissante exhaustive et séparée $\Fil^{j}_n W_n$ de $W_n = K_n \otimes W$
avec $\Fil^{r+1}_n W_n = 0$ telle que
\begin{enumerate}
\item la dimension $v_j$ du $K_n$-espace vectoriel $\Fil^j_n W_n$ est indépendante de $n$ ;
\item pour tout entier $j \leq r$ et toute racine de l'unité $\zeta$ d'ordre $p^n$
$$ g_i(u^{j}\zeta -1) \in \Fil^{j}_n W_n.$$
\end{enumerate}
Le déterminant $\det (g_1, \cdots, g_d)$ calculé dans une base de $W$
est divisible par $\prod_{j \leq r} (\mloginf{[j]}{N})^{d-v_j}$.
\end{lem}
\begin{proof} L'argument est classique (\cite{jams}).
Posons $F = \det (g_1, \cdots, g_d)$ et $ v_j = \dim_{K_n}\Fil^{j}_n W_n$.
Il suffit de montrer que pour $j \leq r$ et $\zeta$ racine de l'unité d'ordre $p^n$,
$F^{(s)}(u^j \zeta -1)$
est nul pour $s < d - v_j$.
Or $F^{(s)}$ s'exprime comme une combinaison linéaire de déterminants faisant intervenir $d-s$ des colonnes $g_i$
donc du type
$G = \det (h_1, \cdots, h_{d-s}, \cdots)$ avec $\{h_1, \cdots, h_{d-s}\} \subset \{g_1, \cdots g_d\}$.
Comme les $h_i(u^j\zeta -1)$ pour $0\leq i \leq d-s$ appartiennent au sous-espace $ \Fil^{j}_n W_n$
de dimension $v_j < d-s$, $G$ est nul. On en déduit que $F^{(s)}(u^j \zeta -1)$
est nul pour $s < d - v_j$ et donc que $F$ est divisible par
$(\mloginf{[j]}{N})^{d-v_j}$.
\end{proof}

\begin{thm} On suppose que $\tborne{\lgg{\JJ}}=0$.
Soit $\cB$ une base de $\DD$ adaptée à la filtration $\Fil^\bullet \DD$.
Alors la suite des diviseurs élémentaires
de $\cH\oEE\otimes\oEE\Module{\JJ}{\cB}{N}{\DD}$ dans $\cH\oEE\otimes\oEE \DD$
est
$$[\mloginf{\III{t_{\HT,1}}}{N}; \cdots; \mloginf{\III{t_{\HT,i}}}{N}; \cdots ;
1].$$
\end{thm}
\begin{proof}
On utilise l'expression de la matrice de $\Zinftyu{\JJ}{\cB}{N}$ dans la base $\cB$
comme produit de matrices de la forme
$$\UU{\cB}{\np}diag([\txin{n}{\inclus{\III{t_{\HT,v}}}{\JJ}}]) \UU{\cB}{\nm}$$
pour $n\geq N$.
La matrice $\VV{\cB}$ est inversible dans $\cH\oEE$
puisque à coefficients dans $\EEE$ et de déterminant non nul. Il suffit donc d'appliquer le
lemme qui suit en remarquant que pour $n\geq 1$, $m \geq 1$,
$\txin{n}{\inclus{\III{t_1}}{\JJ}}$ est premier à $\txin{m}{\inclus{\III{t_2}}{\JJ}}$
si $n \neq m$:
ils sont égaux à une unité près de $\cH$ respectivement
à $\prod_{j \in \III{t_1}} \xi_n^{(j)}$ et à $\prod_{j \in \III{t_2}} \xi_m^{(j)}$
sous l'hypothèse que $\tborne{\lgg{\JJ}}=0$.
\end{proof}
\begin{lem}
Soient $A$ et $B$ deux matrices dont les invariants de Smith sont respectivement
$[f_1; \cdots ; f_d ]$ et $[g_1; \cdots ; g_d ]$. Si les $f_i$ et les $g_j$ sont premiers entre eux,
les invariants de Smith de $C = AB$ sont $[f_1 g_1; \cdots;f_d g_d] $.
\end{lem}
\begin{proof}
On utilise la caractérisation de l'invariant de Smith $f_i \cdots f_d$ comme étant le pgcd
des mineurs d'ordre $d-i+1$ de la matrice $A$. On peut supposer que $A$ est sous forme diagonale
$Diag(f_1,\cdots, f_d)$ quitte à multiplier $B$ à gauche par une matrice
de $GL_d(\cH)$. Si $B=((b_{ij}))$, $AB = ((f_i b_{i,j}))$.
Un mineur d'ordre $i$ de $AB$ est alors le produit de deux mineurs d'ordre $i$
respectivement de $A$ et $B$. On en déduit
que $h_i \cdots h_d$ est le produit de $f_i \cdots f_d$ et de $g_i \cdots g_d$
lorsque les $f_j$ et les $g_k$ sont premiers entre eux. Par induction, on en déduit que $h_j | f_j g_j$.
Il reste à remarquer que le produit des $h_j$ est égal à $\prod_j f_j \prod_j g_j$
pour conclure.
\end{proof}
\begin{reme}
Dans le cas où $\tborne{\lgg{\JJ}}\neq 0$, les $\mxiinf{N}{\inclus{\III{t_{\HT,i}}}{\JJ}}$
ne se divisent pas forcément à cause de facteurs supplémentaires.
Le pgcd de $\mxiinf{N}{\inclus{\III{t_{\HT,i}}}{\JJ}}$ et de
$\mxiinf{N}{\inclus{\III{t_{\HT,j}}}{\JJ}}$
pour $t_{\HT,i}>t_{\HT,j}$ est un multiple de $\mloginf{\III{t_{\HT,i}}}{N}$.
 Expérimentalement, il semble être égal à $\mloginf{\III{t_{\HT,i}}}{N}$.
Ce qui donnerait à espérer que les invariants soient de la forme :
$$[\mloginf{\III{t_{\HT,1}}}{N}; \mloginf{\III{t_{\HT,d-1}}}{N} ; \cdots ;
\mloginf{\III{t_{\HT,j}}}{N}; \cdots ; F]$$
où $F$ est un élément de $\cH$ ne s'annulant pas en les $u^j \zeta -1$ pour
$j \in J$ et $\zeta$ racine de l'unité d'ordre une puissance de $p$.
\end{reme}

\begin{prop}On suppose que $\tborne{\lgg{\JJ}}=0$.
Soit $\cF$ un raffinement du $\varphi$-module filtré $\DD$ et
soient $\cB$ et $\cB_0$ des bases associées à $\cF$ comme dans le paragraphe
\ref{raffinement}.
La suite $(\Zinftyu{\JJ}{\cB}{N}(v))_{v\in \cB_0}$ d'éléments de $\Module{\JJ}{\cB}{N}{\DD}$
est une base adaptée aux diviseurs élémentaires de
$\cH\oEE \otimes\oEE \Module{\JJ}{\cB}{N}{\DD}$ dans $\cH\oEE \otimes\oEE\DD$.
\end{prop}
\begin{proof}
On reprend les notations de la démonstration du théorème \ref{znraff}
et on calcule dans la base $\cB_0$.
On remarque que $\mlog{\III{t_{\HT,i}}}{n}$ divise les $i$ premières colonnes
de $T_{n}$ et donc de $\varphi_{\cB_0}^{\np} T_n \varphi_{\cB_0}^{\nm}$.
La matrice $Z_n$ étant triangulaire supérieure, on en déduit que
les $i$ premières colonnes de la matrice de
$\ZZZ{\JJ}{\cB}{n}{N}$ dans la base $\cB_0$ sont aussi divisibles
par $\mlog{\III{t_{\HT,i}}}{n}$ et qu'il en est de même
de la matrice de $\ZZZ{\JJ}{\cB}{m}{N}$ pour $m\geq n$.
Donc, les $i$ premières colonnes de la matrice de
$\ZZZ{\JJ}{\cB}{m}{N}$ dans la base $\cB_0$
sont divisibles par $\mlog{\III{t_{\HT,i}}}{n}$.
\end{proof}
\begin{thm}\label{prop}
On suppose que $\tborne{\lgg{\JJ}}=0$.
Soit $\cB$ une base de $\DD$ adaptée à la filtration $\Fil^\bullet \DD$.
Alors, $\LLL{\JJ}{N}{\DD}$ est égal à $\cH\oEE \otimes\oEE\Module{\JJ}{\cB}{N}{\DD}$.
En particulier, l'indice de $\LLL{\JJ}{N}{\DD}$ dans $\cH\oEE \otimes\oEE \DD$
est $$\prod_{j\leq t_{\HT,d}} (\mloginf{[j]}{N})^{\dim \DD/\Fil^{j}\DD}$$ et
les diviseurs élémentaires de $\LLL{\JJ}{N}{\DD}$ dans $\cH\oEE \otimes\oEE \DD$
sont $$[\mloginf{\III{t_{\HT,1}}}{N}; \cdots; \mloginf{\III{t_{\HT,j}}}{N}; \cdots ;
 1].$$
\end{thm}
\begin{proof}
Le déterminant de $\Zinftyu{\JJ}{\cB}{N}$ dans $\cH\oEE \otimes\oEE \DD$ est
$\prod_{j\leq t_{\HT,d}} (\mloginf{[j]}{N})^{\dim \DD/\Fil^{j}\DD}$. Comme
$\LLL{\JJ}{N}{\DD}$ contient $\cH \otimes \Module{\JJ}{\cB}{N}{\DD}$, le déterminant
de $\LLL{\JJ}{N}{\DD}$ dans $\cH \otimes \DD$ divise
$\prod_{j\leq t_{\HT,d}} (\mloginf{[j]}{N})^{\dim \DD/\Fil^{j}\DD}$,
d'où l'égalité grâce à la proposition \ref{detdiv},
d'où le théorème.
\end{proof}
\begin{cor}On suppose que $\tborne{\lgg{\JJ}}=0$.
Soit $\cB$ une base adaptée à la filtration et adaptée
à un raffinement. La suite d'éléments
$(\Zinftyu{\JJ}{\cB}{N}(v))_{v\in \cB_0}$ de $\LLL{\JJ}{N}{\DD}$
est une base adaptée aux diviseurs élémentaires de
$\LLL{\JJ}{N}{\DD} \subset \cH\oEE \otimes\oEE \DD$.
\end{cor}
On déduit du théorème \ref{prop} et du corollaire \ref{cor:divisible}
le théorème suivant.
\begin{thm}\label{thmordre}
On suppose que $\tborne{\lgg{\JJ}}=0$. S'il existe une base $\cB$ adaptée au $\varphi$-module filtré
$ \DD$ engendrant un réseau fortement divisible dans $\DD$,
$\LLL{\JJ}{N}{\DD}$ admet une base formée d'éléments
de pente nulle, explicitement $(\Zinftyu{\JJ}{\cB}{N}(v))_{v\in \cB}$.
\end{thm}

\begin{cor}\label{pente0}Sous les hypothèses du théorème \ref{thmordre},
le $\Lambda\ot$-module $\Module{\JJ}{\cB}{N}{\DD}$ est
le sous-$\Lambda\ot$-module
de $\LLL{\JJ}{N}{\DD}$ engendré par les éléments de pente nulle.
\end{cor}
\begin{proof}
Soit $X$ le $\Lambda\ot$-sous-module de $\LLL{\JJ}{N}{\DD}$ des éléments de pente $\leq 0$.
Par le théorème \ref{thmordre}, on a
$\LLL{\JJ}{N}{\DD} = \cH X=\cH\Module{\JJ}{\cB}{N}{\DD}$. Le résultat se déduit alors
de l'inclusion $\Module{\JJ}{\cB}{N}{\DD} \subset X$ et de ce que les éléments inversibles
de $\cH$ sont les éléments inversibles de $\Lambda\ot$.
\end{proof}
\begin{cor} Soit $\zeta$ une racine de l'unité d'ordre $p^n$.
Sous les hypothèses du théorème \ref{thmordre},
l'application d'évaluation
$$\LLL{\JJ}{N}{\DD} \to \prod_{j\in \III{t_{\HT,1}}} K_n \otimes \Fil^{j} \DD$$
induite par
$$ G \mapsto ((1\otimes \varphi)^{\nm}G(u^j \zeta-1) )_{j}$$
et restreinte aux éléments de pente nulle est surjective.
\end{cor}
\bibliographystyle{amsplain}
\selectbiblanguage{french}
\bibliography{biblio}

\providecommand{\bysame}{\leavevmode\hbox to3em{\hrulefill}\thinspace}
\providecommand{\MR}{\relax\ifhmode\unskip\space\fi MR }
\providecommand{\MRhref}[2]{%
  \href{http://www.ams.org/mathscinet-getitem?mr=#1}{#2}
}
\providecommand{\href}[2]{#2}
\begin{thebibliography}{10}

\bibitem{av}
Yvette Amice and Jacques Vélu, \emph{Distributions $p$-adiques associées aux
  séries de {H}ecke}, Astérisque \textbf{24-25} (1975), 119--131.

\bibitem{bell-chen}
Joel Bellaiche and Gaetan Chenevrier, \emph{Families of {G}alois
  representations and {S}elmer groups}, vol. 324, Soc. Math. France, Paris,
  2009.

\bibitem{berger1}
Laurent Berger, \emph{Représentations $p$-adiques et équations
  différentielles}, Invent. math. \textbf{148} (2002), 219--284.

\bibitem{fontaine}
Jean-Marc Fontaine, \emph{Modules galoisiens, modules filtrés et anneaux de
  {B}arsotti-tate}, Astérisque, Journées de Géométrie algébrique de Rennes
  \textbf{65} (1979), 3--80.

\bibitem{FL}
Jean-Marc Fontaine and Guy Laffaille, \emph{Construction de représentations
  $p$-adiques}, Ann. Sci. Ecole Norm. Sup. \textbf{15} (1982), 547--608.

\bibitem{katz}
Nicolas Katz, \emph{Slope filtration of {F}-crystals}, Astérisque, Journées de
  Géométrie algébrique de Rennes \textbf{63} (1979), 113--163.

\bibitem{kedlaya08}
Rikan~S. Kedlaya, \emph{Slope filtration for relative frobenius}, Astérisque
  \textbf{319} (2008), 259--301.

\bibitem{LLZ10}
Antonio Lei, David Loeffler, and Sarah~Livia Zerbes, \emph{{W}ach modules and
  {I}wasawa theory for modular forms}, Asian J. Math. \textbf{14} (2010),
  475--528.

\bibitem{LLZ11}
\bysame, \emph{Coleman maps and the $p$-adic regulator}, Algebra and Number
  Theory \textbf{5} (2011), 1095--1131.

\bibitem{llz17}
\bysame, \emph{On the asymptotic growth of {B}loch-{K}ato-{S}hafarevich-{T}ate
  groups of modular forms over cyclotomic extensions}, Canadian Journal of
  Mathematics \textbf{69} (2017), no.~4, 826--850.

\bibitem{LZ}
David Loeffler and Sarah~Livia Zerbes, \emph{{W}ach modules and critical slope
  $p$-adic $l$ functions}, J. Reine Angew. Math. \textbf{679} (2012), 181--206.

\bibitem{Mazur}
Barry Mazur, \emph{The theme of $p$-adic variation}, Amer. Math.Soc,
  Mathematics: frontiers and perspectives \textbf{63} (2000), 433--459.

\bibitem{bpr93}
Bernadette Perrin-Riou, \emph{Fonctions {$L$} $p$-adiques d'une courbe
  elliptique et points rationnels}, Ann. Inst. Fourier, Grenoble \textbf{43}
  (1993), 945--995.

\bibitem{bpr-debut}
\bysame, \emph{Théorie d'{I}wasawa des représentations $p$-adiques sur un corps
  local (avec un appendice de {J.-M. Fontaine})}, Inventiones mathematicae
  \textbf{115} (1994), no.~1, 81--150.

\bibitem{bpr-colmez}
\bysame, \emph{Théorie d'{I}wasawa et loi explicite de réciprocité : un remake
  d'un article de {P. Colmez}}, Doc. Math. \textbf{4} (1999), 215--269.

\bibitem{jams}
\bysame, \emph{Représentations $p$-adiques et normes universelles: I. {L}e cas
  cristallin}, Journal of the American Mathematical Society \textbf{13} (2000),
  533--551.

\bibitem{semi-stable}
\bysame, \emph{Théorie d'{I}wasawa des représentations $p$-adiques
  semi-stables}, vol.~84, Mém. Soc. Math. Fr., 2001.

\bibitem{divelem}
\bysame, \emph{Note sur les diviseurs élémentaires du régulateur d'{I}wasawa},
  Publications Mathématiques de Bordeaux \textbf{33} (2021), 1069--1075.

\bibitem{pollack}
Robert Pollack, \emph{On the $p$-adic $l$-function of a modular form at a
  supersingular prime}, Duke Mathematical Journal \textbf{118} (2003),
  523--558.

\end{thebibliography}

\end{document}